\documentclass[reqno,12pt,letterpaper]{amsart}

\usepackage{amsmath,amssymb,amsthm,mathrsfs,enumitem}
\usepackage{upgreek}
\usepackage{slashed}
\usepackage{xifthen}
\usepackage{graphicx}
\usepackage{longtable}

\usepackage{xcolor}
\definecolor{winered}{rgb}{0.7,0,0}
\definecolor{lessblue}{rgb}{0,0,0.7}

\usepackage[pdftex,colorlinks=true,linkcolor=winered,citecolor=lessblue,breaklinks=true,bookmarksopen=true]{hyperref}

\usepackage{titletoc}

\makeatletter

\def\@tocline#1#2#3#4#5#6#7{
\begingroup
  \par
    \parindent\z@ \leftskip#3 \relax \advance\leftskip\@tempdima\relax
                  \rightskip\@pnumwidth plus 4em \parfillskip-\@pnumwidth
    \ifcase #1 % sections
       \vskip 0.6em \hskip 0em % add a little vspace before
       \or
       \or \hskip 0em % subsections
       \or \hskip 1em % subsubsections
    \fi%
    #6
    \nobreak\relax{\leavevmode\leaders\hbox{\,.}\hfill}
    \hbox to\@pnumwidth {\@tocpagenum{#7}}
  \par
\endgroup
}

 \def\l@section{\@tocline{0}{0pt}{0pc}{}{}}

\renewcommand{\tocsection}[3]{%
  \indentlabel{\@ifnotempty{#2}{ % for numbered sections
    \ignorespaces\bfseries{#2. #3}}}
  \indentlabel{\@ifempty{#2}{\ignorespaces\bfseries{#3}}{}} % for unnumbered sections
    \vspace{1.5pt}}

\renewcommand{\tocsubsection}[3]{%
  \indentlabel{\@ifnotempty{#2}{
    \ignorespaces#2. #3}}
  \indentlabel{\@ifempty{#2}{\ignorespaces #3}{}}
    \vspace{1.5pt}}

\renewcommand{\tocsubsubsection}[3]{%
  \indentlabel{\@ifnotempty{#2}{
    \ignorespaces#2. #3}}
  \indentlabel{\@ifempty{#2}{\ignorespaces #3}{}}
    \vspace{1.5pt}}

\makeatother

\makeatletter
\def\@nomenstarted{0}
\newlength{\@nomenoldtabcolsep}

\newcommand{\nomenstart}
  {%
    \def\@nomenstarted{1}%
    \setlength{\@nomenoldtabcolsep}{\tabcolsep}%
    \setlength{\tabcolsep}{3.5pt}%
    \begin{longtable}{p{0.11\textwidth} p{0.86\textwidth}}%found by hand
  }

\newcommand{\nomenitem}[2]{%
    \ifcase\@nomenstarted%
      \or % if nomenstarted=1, do nothing
      \or \\ % if nomenstarted=2, add newline to previous one
    \fi%
    #1\,{\leavevmode\leaders\hbox{\,.}\hfill} & #2%
    \def\@nomenstarted{2}%
  }%
\newcommand{\nomenend}
  {\\%
      \end{longtable}%
      \setlength{\tabcolsep}{\@nomenoldtabcolsep}%
      \def\@nomenstarted{0}%
  }
\makeatother

\allowdisplaybreaks

\numberwithin{equation}{section}
\numberwithin{figure}{section}
\newtheorem{thm}{Theorem}[section]
\newtheorem{prop}[thm]{Proposition}
\newtheorem{lemma}[thm]{Lemma}

\newtheorem*{thm*}{Theorem}
\newtheorem*{prop*}{Proposition}
\newtheorem*{cor*}{Corollary}
\newtheorem*{conj*}{Conjecture}

\theoremstyle{definition}
\newtheorem{definition}[thm]{Definition}

\theoremstyle{remark}
\newtheorem{rmk}[thm]{Remark}
\newtheorem{example}[thm]{Example}

\newcommand{\mc}{\mathcal}

\newcommand{\cB}{\mc B}
\newcommand{\cC}{\mc C}
\newcommand{\cD}{\mc D}

\newcommand{\cL}{\mc L}

\newcommand{\cO}{\mc O}

\newcommand{\cS}{\mc S}
\newcommand{\cT}{\mc T}
\newcommand{\cU}{\mc U}
\newcommand{\cV}{\mc V}

\newcommand{\ms}{\mathscr}

\newcommand{\sN}{\ms N}

\newcommand{\sS}{\ms S}
\newcommand{\sV}{\ms V}

\newcommand{\N}{\mathbb{N}}
\newcommand{\R}{\mathbb{R}}

\newcommand{\Sph}{\mathbb{S}}

\newcommand{\sfp}{\mathsf{p}}
\newcommand{\sfs}{\mathsf{s}}

\newcommand{\sfJ}{\mathsf{J}}

\newcommand{\sfV}{\mathsf{V}}

\newcommand{\frakM}{\mathfrak{M}}

\newcommand{\mathspan}{\operatorname{span}}

\newcommand{\sgn}{\operatorname{sgn}}

\newcommand{\expb}{\exp^{\mathrm{b}}}
\newcommand{\Lightb}{L^{\mathrm{b}}}
\newcommand{\Lightbp}{L^{\mathrm{b,+}}}

\newcommand{\reg}{{\mathrm{reg}}}

\newcommand{\eps}{\epsilon}

\newcommand{\Lra}{\Longrightarrow}

\newcommand{\hra}{\hookrightarrow}

\newcommand{\ol}{\overline}
\newcommand{\pa}{\partial}

\newcommand{\ul}{\underline}

\newcommand{\wt}{\widetilde}
\newcommand{\xra}{\xrightarrow}

\newcommand{\bop}{{\mathrm{b}}}

\newcommand{\Vf}{\mathcal V}

\newcommand{\Vb}{\Vf_\bop}

\newcommand{\rcTbdual}[1][]{\ensuremath{\overline{{}^{\bop}T^*\ifthenelse{\isempty{#1}}{}{_#1}}}}

\newcommand{\bhm}[1][]{\ensuremath{M_{\bullet\ifthenelse{\isempty{#1}}{}{,#1}}}}

\newcommand{\CI}{\cC^\infty}

\newcommand{\openbigpmatrix}[1]
  {%
    \def\@bigpmatrixsize{#1}%
    \addtolength{\arraycolsep}{-#1}%
    \begin{pmatrix}%
  }
\newcommand{\closebigpmatrix}
  {%
    \end{pmatrix}%
    \addtolength{\arraycolsep}{\@bigpmatrixsize}%
  }

\newcommand{\itref}[1]{(\ref{#1})}

\DeclareGraphicsExtensions{.mps}
\newcommand{\inclfig}[1]{\includegraphics{#1}}

\begin{document}

%%%%%%%%%%%%%%%%%%%%%%%%%%%%%%%%%%%%%%%%%%%%%%%%%%%%%%%%%%%%%%%%%%%%%%
% title page
\title[Reconstruction from boundary light observation sets]{Reconstruction of Lorentzian manifolds from boundary light observation sets}

\author{Peter Hintz}
\address{Department of Mathematics, University of California, Berkeley, CA 94720-3840, USA}
\email{phintz@berkeley.edu}

\author{Gunther Uhlmann}
\address{Department of Mathematics, University of Washington, Seattle, WA 98195-4350, USA}
\address{Institute for Advanced Study of the Hong Kong University of Science and Technology, Hong Kong, China}
\address{Department of Mathematics and Statistics, University of Helsinki, Finland}
\email{gunther@math.washington.edu}

\date{May 18, 2017. Final revision: October 8, 2017.}


\begin{abstract}
  On a time-oriented Lorentzian manifold $(M,g)$ with non-empty boundary satisfying a convexity assumption, we show that the topological, differentiable, and conformal structure of suitable subsets $S\subset M$ of sources is uniquely determined by measurements of the intersection of future light cones from points in $S$ with a fixed open subset of the boundary of $M$; here, light rays are reflected at $\pa M$ according to Snell's law. Our proof is constructive, and allows for interior conjugate points as well as multiply reflected and self-intersecting light cones.
\end{abstract}

\maketitle

%%%%%%%%%%%%%%%%%%%%%%%%%%%%%%%%%%%%%%%%%%%%%%%%%%%%%%%%%%%%%%%%%%%%%%
\section{Introduction}
\label{SecIntro}

Let $(M,g)$ be a Lorentzian manifold with a non-empty boundary, and denote by $M^\circ$ its interior. We consider the problem of reconstructing the topological, differentiable, and conformal structure of subsets $S\subset M^\circ$ by boundary observations of light cones emanating from points in $S$, with light rays being reflected at $\pa M$ according to Snell's law. We accomplish this under a convexity assumption on $\pa M$ and assuming that broken (reflected) null-geodesics from $S$ have no conjugate points lying on $\pa M$. The present paper is similar in spirit to the work by Kurylev, Lassas, and Uhlmann \cite{KurylevLassasUhlmannSpacetime}: they consider a related reconstruction problem using light observation sets in the \emph{interior} of globally hyperbolic spacetimes without boundary. The presence of a boundary leads to a much richer structure of the broken null-geodesic flow, and observing only \emph{at} the boundary limits the available leeway when light cones are singular (conjugate points or self-intersections) at $\pa M$.

To state a simple example to which our main result, stated below, applies, consider the manifold $M=\{(t,x)\in\R^{1+2}\colon |x|<1\}$, equipped with the Minkowski metric $g=-dt^2+dx^2$, and let the set $S$ of sources be an open subset $S\subseteq\{(t,x)\colon|t|<1/2-|x|\}\subset M$. The boundary light observation set from a point $q=(t_0,x_0)\in S$ within the subset $\cU:=\{(t,x)\colon 0<t<2,\ |x|=1\}\subset\pa M$ is the intersection
\[
  \cL^+_q \cap \cU = \{ (t,x)\in\cU \colon t\geq t_0,\ t-t_0=|x-x_0| \}
\]
of the future light cone from $q$ with $\cU$. See Figure~\ref{FigIntroEx}. Let $\cS=\{\cL^+_q\cap\cU\colon q\in S\}$ denote the family (as an unlabelled set) of boundary light observation sets. Then from $\cS$, we can reconstruct $S$ as a smooth manifold, as well as the conformal class of the metric $g|_S$.

This example generalizes in a straightforward manner to higher dimensions; in $1+3$ dimensions, this would be a very simple model for wave propagation in the interior $|x|<1$, $x\in\R^3$, of the Earth, with observations taking place for some limited period of time on the surface of the Earth. More generally, our main theorem allows the wave speed to be inhomogeneous, anisotropic, and time-dependent.

\begin{figure}[!ht]
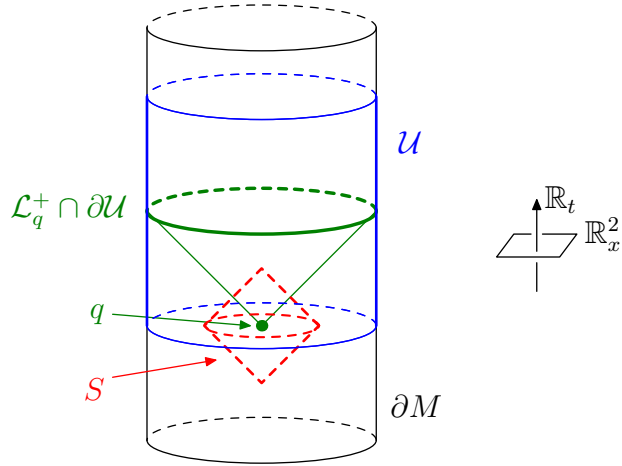

\centering
  \inclfig{IntroEx}
  \caption{One can recover the topological, differentiable, and conformal structure of $S$ from the collection of boundary light observation sets.}
\label{FigIntroEx}
\end{figure}

In general, the future light cone $\cL^+_q$ from a point $q\in M$ is defined as the union of all future-directed broken null-geodesics. (See Figure~\ref{FigGeoGBroken} for an illustration, and Definition~\ref{DefGeoG} for the precise definition.) Our main theorem applies to rather general Lorentzian manifolds, and allows for the reconstruction of $S$ from boundary light observation sets involving multiple reflections. (See Remark~\ref{RmkRUsefulRefl}.) To set this up, we define the class of manifolds we will work with:

\begin{definition}
\label{DefIntroAdm}
  Let $n\geq 1$. Let $(M,g)$ be a smooth connected $(n+1)$-dimensional Lorentzian manifold with non-empty boundary; thus, $g$ has signature $(-,+,\ldots,+)$. We call $(M,g)$ \emph{admissible} if
  \begin{enumerate}
    \item there exists a proper, surjective function $t\colon M\to\R$ such that $dt$ is everywhere timelike;
    \item the boundary $\pa M$ is timelike, i.e.\ the induced metric $g_\pa:=g|_{\pa M}$ is Lorentzian;
    \item\label{ItIntroAdmNullConvex} $\pa M$ is \emph{null-convex}: if $\nu$ denotes the outward pointing unit normal vector field on $\pa M$, then
    \begin{equation}
    \label{EqIntroAdmNull}
      II(V,V) = g(\nabla_V\nu, V) \geq 0
    \end{equation}
    for all null vectors $V\in T_p\pa M$.
  \end{enumerate}
\end{definition}

We recall that a vector $V\in T_p M$ in a Lorentzian manifold $(M,g)$ is called \emph{timelike}, \emph{spacelike}, or \emph{lightlike} (\emph{null}) whenever $g_p(V,V)<0$, $g_p(V,V)>0$, or $g_p(V,V)=0$, respectively. An admissible manifold $(M,g)$ is time orientable, as we can declare $dt$ to be \emph{past} timelike. (We refer the reader to \cite{ONeillSemi} for further background on Lorentzian geometry.) If $n=1$, then condition~\itref{ItIntroAdmNullConvex} is vacuous.

For the purposes of this introduction, we will work with manifolds $(M,g)$ with \emph{strictly null-convex} boundaries, that is, \eqref{EqIntroAdmNull} holds with \emph{strict} inequality for $V\neq 0$. In this case, all broken null-geodesics are well-defined globally on $M$, see \S\ref{SubsecGeoTame}.

\begin{thm}
\label{ThmIntro}
  Let $(M_j,g_j)$, $j=1,2$, be two admissible Lorentzian manifolds with strictly null-convex boundaries, let $S_j\subset M_j^\circ$ be open with compact closure in $M_j$, and let $\cU_j'\Subset\cU_j\subset\pa M_j$ be open. Let
  \[
    \cS_j:=\{\cL^+_q\cap\cU_j\colon q\in S_j\}.
  \]
  Assume that for $q_1,q_2\in\ol{S_j}$, the equality of boundary light observation sets $\cL_{q_1}^+\cap\cU'_j=\cL_{q_2}^+\cap\cU'_j$ implies $q_1=q_2$. Assume moreover that for $q\in S_j$, no point in $\cU_j$ which lies on a future-directed broken null-geodesic starting at $q$ is conjugate to $q$.

  Suppose there exists a diffeomorphism $\Phi\colon\cU_1\xra{\cong}\cU_2$ which identifies the families of boundary light observation sets, that is, $\cS_2=\{\Phi(L)\colon L\in\cS_1\}$. Then there exists a conformal diffeomorphism $\Psi\colon(S_1,g_1|_{S_1})\xra{\cong}(S_2,g_2|_{S_2})$.

  If in addition $\Phi$ is conformal for the metrics $g_j|_{\cU_j}$ on $\cU_j$ and time orientation preserving, then $\Psi$ preserves the time orientation as well.
\end{thm}

Thus, if the smooth structure of the observation set $\cU_j$ is given, then the collection of light observation sets --- carrying no structure other than that of a set! --- uniquely determines the topological, differentiable, and conformal structure of the set of sources; given a conformal structure and time orientation on $\cU_j$, one can in addition recover the time orientation of the set of sources. See Theorem~\ref{ThmR} for the full statement which replaces the strict null-convexity condition with a certain non-degeneracy condition (called \emph{tameness} in \S\ref{SubsecGeoTame}) on broken null-geodesics.

The proof of Theorem~\ref{ThmIntro} proceeds in three steps. First, we define a topology on $\cS_j$ by declaring collections of boundary light observation sets to be open if they intersect, resp.\ miss, a fixed open, resp.\ compact, subset of $\cU$: this topology is shown to be equal to the subspace topology of $S_j$ via the bijection $S_j\ni q\mapsto\cL_q^+\cap\cU_j$; see \S\ref{SubsecRT}. Second, we show how to construct (intrinsically within $\cS_j$ and $\cU_j$) a large class of functions which are smooth on $S_j$: these functions $x^\mu$ assign to a point $q'$ close to a fixed point $q$ the unique parameter $x^\mu(q')$ along suitable curves $\mu\subset\cU_j$ at which $\mu$ intersects $\cL_{q'}^+$. (In \cite{KurylevLassasUhlmannSpacetime}, a similar construction was used \emph{globally}.) We show that \emph{all} smooth functions on $S_j$ are, locally, $\CI$ functions of these $x^\mu$ for varying $q$ and $\mu$; see \S\ref{SubsecRS}. In order to reconstruct the conformal class of $g_j$ on $S_j$, we show how to identify a large number of null-geodesics $s\mapsto q(s)$ in $S_j$ in terms of the boundary light observation sets of the points $q(s)$; see \S\ref{SubsecRC}. Since light cones are well-defined given merely the conformal class of a Lorentzian metric, one can in general not recover the metric itself. (Under additional assumptions, this may be possible, see~\cite[Corollary~1.3]{KurylevLassasUhlmannSpacetime}.) Finally, the time orientation on $S_j$ can be determined by analyzing the behavior of $\cL_q^+\cap\cU_j$ as $q$ moves along a timelike curve in $S_j$; see \S\ref{SubsecRO}.

It would be interesting to reconstruct suitable subsets of $(M,g)$ from \emph{active measurements}, namely from the Dirichlet-to-Neumann map of initial boundary value problems for non-linear wave equations. (In the boundary-less setting, the analogous inverse problem was first solved in the context of the quasilinear Einstein equation \cite{KurylevLassasUhlmannEinstein}, see also \cite{KurylevLassasUhlmannSpacetime}, with improvements by Lassas, Uhlmann, and Wang \cite{LassasUhlmannWangSemi,LassasUhlmannWangEinsteinMaxwell}.) The idea is to generate singular small amplitude distorted plane waves by imposing suitable singular Dirichlet data: these can be engineered so that their non-linear interaction generates point sources at points $q\in M^\circ$, allowing one to identify the boundary light observation set $\cL_q^+\cap\pa M$ by measuring singularities of the Neumann trace; this puts one into the setting of Theorem~\ref{ThmR}. We hope to address this problem in future work. See also~\cite{BelishevKurylevBCMethod,EskinOpticalBlackHoles,LassasOksanenDNDisjoint} for results in related contexts.

For further results on the reconstruction of Lorentzian manifolds, we mention Larsson's work \cite{LarssonSkyShadow} using broken causal lens data or sky shadow data (see also the related \cite{KurylevLassasUhlmannBrokenGeodesicRigidity}), and the work by Lassas, Oksanen, and Yang \cite{LassasOksananYangTimeMeasurements} on the reconstruction of the jet of a Lorentzian metric on a timelike hypersurface from time measurements. There is a large amount of literature on inverse problems on Riemannian manifolds with boundary; we refer to \cite{PestovUhlmann2DSimple,StefanovUhlmannVasyBoundaryRidigity} and the references therein.

The plan of the paper is as follows: in \S\ref{SubsecGeoA}, we analyze the properties of admissible Lorentzian manifolds and give an equivalent formulation of the null-convexity assumption; in \S\ref{SubsecGeoG}, we define the broken null-geodesic flow and discuss its basic properties. We introduce the important notion of \emph{tameness} in \S\ref{SubsecGeoTame}; on admissible manifolds with strictly null-convex boundary, all broken null-geodesics are tame. In \S\ref{SecR} finally, we prove the main result, Theorem~\ref{ThmR}, following the steps outlined above.

%%%%%%%%%%%%%%%%%%%%%%%%%%%%%%%%%%%%%%%%%%%%%%%%%%%%%%%%%%%%%%%%%%%%%%
\section{Geometric preliminaries}
\label{SecGeo}

%%%%%%%%%%%%%%%%%%%%%%%%%%%%%%%%%%%%%%%%%%%%%%%%%%
\subsection{Structure of admissible manifolds}
\label{SubsecGeoA}

We begin by elucidating the smooth structure of admissible manifolds, see Definition~\ref{DefIntroAdm}. We use the notation $\Vb(M)$ for the space of smooth vector fields on $M$ which are tangent to the boundary $\pa M$.

\begin{lemma}
\label{LemmaGeoATime}
  Let $(M,g)$ be an admissible Lorentzian manifold. Then $X:=\{t=0\}$ is a compact submanifold with boundary $\pa X\subset\pa M$, and there exists a diffeomorphism $M\cong\R_t\times X$. Furthermore, there exists a global future timelike vector field $T\in\Vb(M)$ such that $T t=1$.
\end{lemma}
\begin{proof}
  Since $t$ is proper with $dt\neq 0$, the first claim is immediate. Moreover, the time orientation on $M$ induces a time orientation on $\pa M$, since the latter is assumed to be Lorentzian; with this time orientation, $dt|_{\pa M}$ is past timelike.

  Since for $O\subset M$ open the set of future timelike vector fields $V\in\Vb(O)$ with $V t=1$ is convex, it suffices to construct $T$ locally. In the interior of $M$, this is straightforward. In a neighborhood $O$ of a point $p\in\pa M$, one first constructs $T'\in\cV(O\cap\pa M)$ with $T' t=1$; one then extends $T'$ arbitrarily to a vector field $\wt T\in\Vb(O)$, which thus satisfies $\wt T t>1/2$ in a smaller neighborhood $O'\subset O$ of $p$, thus $T=(\wt T t)^{-1}\wt T\in\Vb(O')$ is the desired vector field near $p$.

  The flow $\phi\colon\R\times X\in(s,x)\mapsto\exp_x(s T)\in M$ exists globally; indeed, $t(\exp_x(s T))=s$ for all $(s,x)$, since this holds for $s=0$, and the $s$-derivative of both sides is equal to $1$ by construction. The inverse of $\phi$ is given by $\phi^{-1}(p)=(a,\exp_p(-a T))$ when $p\in t^{-1}(a)$. Thus, $\phi$ establishes a diffeomorphism $\R\times X\cong M$.
\end{proof}

It will be useful to embed $(M,g)$ into a larger spacetime without boundary.
\begin{lemma}
\label{LemmaGeoAEmbed}
  There exists a time-oriented smooth Lorentzian manifold $(\wt M,\wt g)$ into which $M$ embeds isometrically as a submanifold with boundary.
\end{lemma}
\begin{proof}
  Let $M'$ be any open manifold into which $M$ embeds as a submanifold with boundary, e.g.\ take $M'$ to be the double of $M$. Extend $g$ to a symmetric 2-tensor $\wt g$ on $M'$, and extend $t$ to an arbitrary smooth function, still denoted $t$, on $M'$. Since the set of Lorentzian metrics on a fixed vector space is open, and since the condition that $dt$ is timelike (in particular $dt\neq 0$) is open, there exists an open neighborhood $\wt M$ of $M$ on which $\wt g$ is Lorentzian and $dt$ timelike; declaring $dt$ to be \emph{past} timelike endows $\wt M$ with a time orientation.
\end{proof}

Write $\wt\exp$ for the exponential map on $(\wt M,\wt g)$. Denote by $g^+$ a fixed smooth Riemannian metric on $(\wt M,\wt g)$, and write
\[
  |V|_{g^+}:=g^+(V,V)^{1/2}.
\]
(All our arguments will take place in compact subsets of $M$, hence the concrete choice of $g^+$ will be irrelevant.)

We now analyze the null-convexity condition. (We encourage the reader to keep the simpler case in mind that the boundary is \emph{strictly} null-convex.) We introduce the outward ($+$) and inward ($-$) pointing tangent bundles
\[
  T^\pm_{\pa M}M = \{ V\in T_{\pa M}M \colon \pm g(V,\nu) > 0 \},
\]
where $\nu$ is the outward pointing unit normal. Thus, $dx(\nu)<0$ for any boundary defining function $x$ (that is, $x=0$ and $dx\neq 0$ at $\pa M$, while $x>0$ in $M^\circ$), and we therefore also have
\[
  T^\pm_{\pa M}M = \{ V\in T_{\pa M}M \colon \mp dx(V) > 0 \}.
\]
Define the future/past light cones
\[
  L^\pm_p M = \{ V\in T_p M \colon V\ \text{is future ($+$), resp.\ past ($-$), lightlike} \},
\]
and the light cone
\[
  L_p M = \{ V\in T_p M \colon V\ \text{is lightlike} \} \cup \{ 0 \} \subset T_p M.
\]
As a first step, we show:

\begin{lemma}
\label{LemmaGeoABdy}
  Let $(M,g)$ be a Lorentzian manifold with null-convex timelike boundary $\pa M$ and outward pointing unit normal $\nu$. Let $p\in\pa M$. Then there exists $s_0>0$ such that for all lightlike $V\in L_p M$, $|V|_{g^+}=1$, the following holds for the null-geodesic $\gamma(s):=\wt\exp_p(s V)$:
  \begin{enumerate}
    \item\label{ItGeoABdyOut} If $V\in T^\pm_{\pa M}M$, then $\gamma(s)\in\wt M\setminus M$ for $0<\pm s\leq s_0$.
    \item\label{ItGeoABdyTgt} If $V\in T_p\pa M$ is tangent to $\pa M$, then $\gamma(s)\in\wt M\setminus M^\circ$ for $|s|\leq s_0$.
  \end{enumerate}
\end{lemma}
\begin{proof}
  Pick a boundary defining function $x\in\CI(\wt M)$, so $x^{-1}(0)=\pa M$ and $dx\neq 0$ on $\pa M$, and $x>0$ in $M^\circ$, while $x<0$ in $\wt M\setminus M$. Since the outward pointing unit normal to $\pa M$ is then given by $\nu=-|\nabla x|^{-1}\nabla x$, one computes
  \begin{equation}
  \label{EqGeoABdyII}
    II(V,W) = -|\nabla x|^{-1}(H x)(V,W),\ \ V,W\in T\pa M,
  \end{equation}
  where $H x=\nabla^2 x$ is the Hessian of $x$ with respect to $\wt g$. Therefore, the null-convexity condition is equivalent to $(H x)(V,V)\leq 0$ for all $V\in L\pa M=\bigsqcup_{p\in\pa M}L_p\pa M$.

  Denote by $y_1,\ldots,y_n$ smooth coordinates on a neighborhood $U_\pa\subset\pa M$ of $p$, with $y_j=0$ at $p$ for $1\leq j\leq n$. Using a collar neighborhood of $\pa M$, identify the set $U:=U_\pa\times(-x_0,x_0)_x$ (with $x_0>0$ small) with a neighborhood of $p$ in $\wt M$. We will construct a foliation of a small neighborhood of $p$ intersected with $U\cap\{x<0\}$ by \emph{strictly} null-convex hypersurfaces which will act as barriers for the geodesic $\gamma$, roughly speaking preventing it from crossing $\pa M$ into $M^\circ$ too quickly.

   To construct the foliation, let $\delta\in(0,x_0)$ and define the function
   \[
     x_\eps := x + \eps(1-\delta^{-2}Y^2)\ \ \text{for}\ \  Y:=\Bigl(\sum_{i=1}^n y_i^2\Bigr)^{1/2}<\delta^2,\ 0\leq\eps<\delta.
   \]
   We claim that for $\delta>0$ sufficiently small, the level sets $D_\eps:=x_\eps^{-1}(0)$ are strictly null-convex for $\eps>0$. To see this, note that the conormal $dx_\eps=dx-2\eps\delta^{-2}\sum_{i=1}^n y_i dy_i$ of $D_\eps$ is $\eps\delta^{-2}Y$-close (with respect to $g^+$) to $dx$; furthermore, on $D_\eps$, we have $x\in[-\eps,0]$. Given the bound we are imposing on $Y$, we conclude that null vectors $W\in L_{(x,y_1,\ldots,y_n)}D_\eps$ with $|W|_{g^+}=1$ are $\eps$-close to the boundary light cone $L_{(0,y_1,\ldots,y_n)}\pa M$. Since $\{x=0\}$ is null-convex, this implies that
   \[
     (H x)(W,W) \leq C_1\eps
   \]
   for some constant $C_1$. Furthermore, we have $\sum_{i=1}^n dy_i(W)^2\geq C_2>0$ for such $W$ provided $\delta>0$ is sufficiently small. Therefore,
   \begin{align*}
     (H x_\eps)(W,W) &= (H x)(W,W) - 2\eps\delta^{-2}\sum_{i=1}^n dy_i(W)^2 - 2\eps\delta^{-2}\sum_{i=1}^n y_i(H y_i)(W,W) \\
       &\leq C_1\eps - 2\eps\delta^{-2}C_2 - C_3\eps\delta^{-2}Y \\
       &\leq -C_2\eps\delta^{-2}
   \end{align*}
   for sufficiently small $\delta>0$, proving the strict null-convexity of $D_\eps$. Fixing such a $\delta>0$, define
   \[
     \cB:=\bigl\{-\delta/2<x\leq 0,\ Y<\delta^2\bigr\}\subset U,
   \]
   and consider the function
   \[
     f := \frac{x}{1-\delta^{-2}Y^2}
   \]
   on $\cB$, so $D_\eps=f^{-1}(-\eps)$; since $df\neq 0$ is inward pointing at $p$ (indeed, $df=dx$ there), formula~\eqref{EqGeoABdyII} shows that $(H f)(W,W)<0$ for $0\neq W\in L_p D_\eps$, $\eps>0$. See also Figure~\ref{FigGeoABdy}.

   \begin{figure}[!ht]
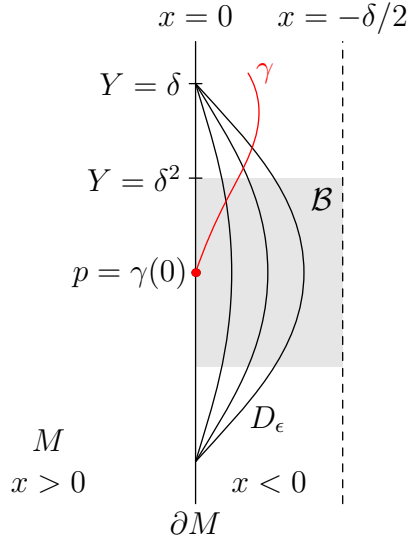

   \centering
     \inclfig{GeoABdy}
     \caption{Foliation of a neighborhood of $p\in\pa M$ in $\wt M\setminus M$ by hypersurfaces $D_\eps$, $\eps>0$, which are \emph{strictly} null-convex. These hypersurfaces are barriers for null-geodesics in the set $\cB$: a null-geodesic $\gamma$ in $\cB$ which emanates from a point in $\pa M$ and has outward pointing initial velocity, that is, $(x\circ\gamma)'(0)<0$, cannot cross $D_\eps$ in the inward direction while in $\cB$.}
   \label{FigGeoABdy}
   \end{figure}
   Consider now $V\in L_p M\cap T^+_{\pa M}M$, $|V|_{g^+}=1$, $\gamma(s)=\wt\exp_p(s V)$. The point of the above construction is that the function $d(s):=f(\gamma(s))$ is negative and strictly decreasing for $s>0$ as long as $\gamma(s)\in\cB$. Indeed, note first that we have $d(0)=0$ and $d'(0)<0$, hence $d(s),d'(s)<0$ for small $s>0$. Suppose now that $d'(s)$ vanishes for some $s>0$ with $\gamma(s)\in\cB$, and let $s'>0$ be the first zero of $d'(s')=0$. Then, letting $\eps:=d(s')$, we have $\gamma'(s')\in L_{\gamma(s')}D_\eps$. The strict null-convexity of $D_\eps$ forces $d''(s')<0$, so $d'(s)$ is strictly decreasing near $s'$; since $d'(s)<0$ for $s<s'$, this contradicts the assumption that $d'(s')=0$.

   Therefore, we have
   \[
     s_0:=\inf_{\genfrac{}{}{0pt}{}{V\in L_p M\cap T^+_{\pa M}M}{|V|_{g^+}=1}}\sup\{s>0\colon \wt\exp_p((0,s]V)\subset\cB^\circ\}>0,
   \]
   where $\cB^\circ$ denotes the interior of $\cB$. (In fact, our arguments show $s_0\gtrsim\delta^2$.) The conclusion of part~\itref{ItGeoABdyOut} then holds for this value of $s_0$.

   Part~\itref{ItGeoABdyTgt} follows from part~\itref{ItGeoABdyOut} by a simple limiting argument: let $V_\eps:=V+\eps\nu$, $|\eps|<1$, which is outward pointing for $\eps>0$ and inward pointing for $\eps<0$. By part~\itref{ItGeoABdyOut}, there exists $s_0>0$ such that $\wt\exp_p(s V_\eps)\in\wt M\setminus M$ for $0<(\sgn\eps)s\leq s_0$. Letting $\eps\to 0$, this implies $\gamma(s)\in\ol{\wt M\setminus M}=\wt M\setminus M^\circ$ for $0\leq|s|\leq s_0$, as claimed.
\end{proof}

We can now give a useful equivalent formulation of the null-convexity condition.

\begin{prop}
\label{PropGeoANull}
  Let $(M,g)$ be a Lorentzian manifold with timelike boundary $\pa M$ and outward pointing unit normal $\nu$. Then the following are equivalent:
  \begin{enumerate}
    \item\label{ItGeoANullConvex} $\pa M$ is null-convex, i.e.\ the inequality~\eqref{EqIntroAdmNull} holds.
    \item\label{ItGeoANullOutward} If $\gamma\colon(-\eps,0]\to M$ is a null-geodesic segment with $\gamma(0)\in\pa M$ and $\gamma(s)\in M^\circ$ for $s\leq 0$, then $\gamma'(0)\in T^+_{\pa M}M$. Likewise, if $\gamma\colon[0,\eps)\to M$ is a null-geodesic segment with $\gamma(0)\in\pa M$ and $\gamma(s)\in M^\circ$ for $s>0$, then $\gamma'(0)\in T^-_{\pa M}M$.
  \end{enumerate}
\end{prop}
\begin{proof}
  \ul{\itref{ItGeoANullConvex} $\Lra$ \itref{ItGeoANullOutward}}: for a null-geodesic $\gamma\colon(-\eps,0]\to M$ as in \itref{ItGeoANullOutward}, the conclusion $\gamma'(0)\in T^+_{\pa M}M\cup T\pa M\setminus\{0\}$ is clear. But by Lemma~\ref{LemmaGeoABdy}, which uses condition~\itref{ItGeoANullConvex}, $\gamma'(0)\in T\pa M$ would imply that $\gamma(s)\in\wt M\setminus M^\circ$ for small $s$. Hence $\gamma'(0)\not\in T\pa M$.

  \ul{\itref{ItGeoANullOutward} $\Lra$ \itref{ItGeoANullConvex}}: suppose that condition~\itref{ItGeoANullConvex} is \emph{violated}, hence there exists $V\in L_p\pa M$, $p\in\pa M$, with $II(V,V)<0$, in particular $V\neq 0$. Define $\gamma(s)=\exp_p(s V)$ for $s\in[0,\eps)$, $\eps>0$ small, and let $f=x\circ\gamma\colon[0,\eps)\to\R$, with $x$ a boundary defining function as in the proof of Lemma~\ref{LemmaGeoABdy}. Then
  \[
    f(0)=0,\ \ f'(0)=0,\ \ f''(0)=(H x)(V,V)>0.
  \]
  Therefore, $\gamma(s)\in M^\circ$ for $s\in(0,s_1)$ for sufficiently small $s_1\in(0,\eps)$. Since $\gamma'(0)=V\in T\pa M$, this contradicts condition~\itref{ItGeoANullOutward}.
\end{proof}

We end this section with a geometric lemma linking boundary light observation sets with spacetime light cones on an infinitesimal level. We denote by $\rho(V)$, $V\in T_p M$, $p\in\pa M$, the reflection of $V$ across $\pa M$, that is,
\begin{equation}
\label{EqGeoARefl}
  \rho(V) := V - 2 g(V,\nu)\nu,
\end{equation}
with $\nu$ the outward pointing unit normal. One easily checks $\rho\colon L M\to L M$. Moreover, if $T\in T\pa M$, then $g(\rho(V),T)=g(V,T)$; this in particular applies to future timelike $T$, hence $\rho\colon L^\pm M\to L^\pm M$ preserves the time orientation of lightlike vectors.

\begin{lemma}
\label{LemmaGeoANull}
  Suppose $(M,g)$ is a time-oriented manifold with timelike boundary $\pa M$. Let $p\in\pa M$. Then there exists an isomorphism $\phi$ between the space $\sS$ of linear spacelike hypersurfaces $S\subset T_p\pa M$ and the space $\sV$ of rays $\R_+ V\subset T_p M$ along future-directed outward pointing null vectors, given by mapping $S\in\sS$ to the unique future-directed outward pointing null ray $\phi(S)$ contained in $S^\perp$. The inverse map is given by $\sV\ni\R_+ V\mapsto T_p\pa M\cap V^\perp\in\sS$.

  Moreover, there exists an isomorphism between $\sS$ and the space $\sN$ of linear null hypersurfaces $N\subset T_p M$ which contain a future-directed outward pointing null vector, given by $\sS\ni S\mapsto S\oplus\mathspan \phi(S)\in\sN$.
\end{lemma}

See Figure~\ref{FigGeoANull}.

\begin{figure}[!ht]
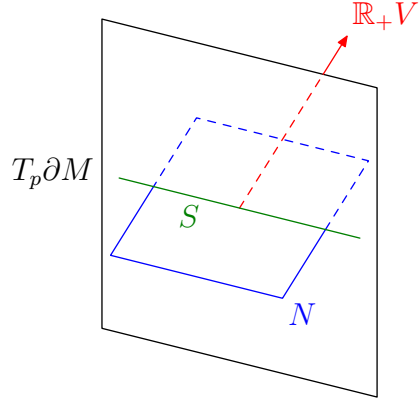

\centering
  \inclfig{GeoANull}
  \caption{Illustration of Lemma~\ref{LemmaGeoANull}; everything takes place in $T_p M$. $N$ is a null hypersurface containing an outward pointing null vector $V$, while $S$ is a spacelike hypersurface in $T_p\pa M$. We have $\R_+V=\phi(S)$, $N^\perp=\R V$, and $S=N\cap T_p\pa M$.}
\label{FigGeoANull}
\end{figure}

\begin{proof}[Proof of Lemma~\ref{LemmaGeoANull}]
  Given a spacelike hypersurface $S\subset T_p\pa M$, the orthocomplement $S^\perp$ is a time-oriented 2-dimensional vector space with signature $(1,1)$, hence there exists a non-zero null vector $W\in S^\perp$; the four distinct rays of null vectors contained in $S^\perp$ are then the positive scalar multiples of $W$, $-W$, $\rho(W)$, $-\rho(W)$. Since multiplication by $-1$ exchanges future- and past-directed null as well as outward and inward pointing vectors, and since application of $\rho$ exchanges outward and inward pointing vectors but preserves the time orientation, exactly one of these four vectors, which we call $V$, is future-directed and outward pointing; and $\phi^+(S)=\R_+ V$.

  On the other hand, if $0\neq V\in T_p M$ is null (thus $V^\perp/\R V$ is spacelike) and outward pointing, in particular $V\not\in T_p\pa M$, then the composition $V^\perp\cap T_p\pa M\hra V^\perp\to V^\perp/\R V$ is an isometric isomorphism, hence $S:=V^\perp\cap T_p\pa M$ is a spacelike hypersurface. This establishes the isomorphism $\sS\cong\sV^+$ (as smooth manifolds).

  For the last claim, we note that $\sN^+\ni N\mapsto N^\perp\cap L^+_p M\in\sV^+$ maps a null hypersurface $N$ into the unique ray along a future-directed outward pointing null generator of $N$. The inverse of this map is given by $\sV^+\ni\R_+ V\mapsto V^\perp=V\oplus(V^\perp\cap T_p\pa M)\in\sN^+$. Composition of these maps with $\phi^+$ gives the desired isomorphism $\sS\xra{\cong}\sN^+$. The inverse of this isomorphism is given by $N\mapsto N\cap T_p\pa M$.
\end{proof}

%%%%%%%%%%%%%%%%%%%%%%%%%%%%%%%%%%%%%%%%%%%%%%%%%%
\subsection{Examples of admissible manifolds}
\label{SubsecGEx}

Small perturbations of admissible Lo\-rentz\-ian manifolds with \emph{strictly} null-convex boundaries are admissible:

\begin{lemma}
\label{LemmaGExOp}
  Suppose $(M,g)$ is admissible and strictly null-convex, with an embedding $(M,g)\hra(\wt M,\wt g)$ as in Lemma~\ref{LemmaGeoAEmbed}. Let $K\Subset\wt M$, and define $\cC^k$ spaces using the Riemannian metric $g^+$ on $\wt M$.
  \begin{enumerate}
    \item Let $x\in\CI(\wt M)$ denote a defining function of $\pa M$. If $x'\in\CI(\wt M)$ is equal to $x$ outside of $K$ and sufficiently close in $\cC^2$ to $x$ in $K$, then $M':=\{x'\geq 0\}$ is admissible and strictly null-convex.
    \item If $g'$ is a smooth Lorentzian metric on $M$, equal to $g$ on $M\setminus K$ and sufficiently close in $\cC^1(K)$ to $g$, then $(M,g')$ is admissible.
  \end{enumerate}
\end{lemma}
\begin{proof}
  This follows from the observation that the assumption of strict null-convexity involves up to first derivatives of the metric and up to second derivatives of the boundary defining function, see \eqref{EqGeoABdyII}.
\end{proof}

If more is known about the global structure of $(M,g)$, one can allow non-compact perturbations as well. For example, the cylinder
\begin{equation}
\label{EqGExM0}
  M_0 := \{ (t,x) \in \R\times\R^n \colon |x|<R \}, \ \ R>0,
\end{equation}
with the Minkowski metric $g=-dt^2+dx^2$, is admissible with strictly null-convex boundary; indeed,
\[
  II(V,V) = R^{-1}dx(V)^2
\]
is strictly positive for non-zero null vectors $V\in T\pa M$. If now $f\colon\R^{1+n}\to\R$ has small $\cC^2$ norm, then
\begin{equation}
\label{EqGExMf}
  M_f := \{ (t,r\omega) \in \R\times\R^n \colon r<(1+f(t,\omega))R \}
\end{equation}
is admissible, with strictly null-convex boundary; see Figure~\ref{FigGExAdm}.

\begin{figure}[!ht]
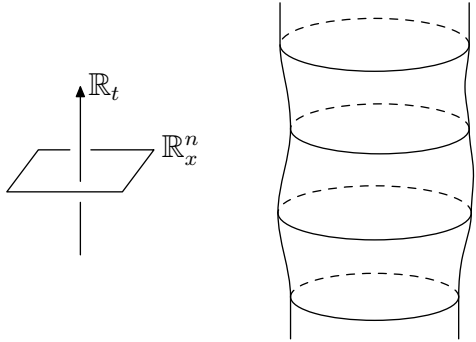

\centering
  \inclfig{GExAdm}
  \caption{An admissible manifold $(M_f,g)\hra(\R^{1+n}_{t,x},-dt^2+dx^2)$, with $f$ having small $\cC^2$ norm.}
\label{FigGExAdm}
\end{figure}

Another interesting class of examples, which includes the cylinder~\eqref{EqGExM0}, is obtained as follows: let $(X,h)$ be a compact Riemannian manifold with convex boundary, so $II(V,V)=h(\nabla_V \nu,V)\geq 0$ for all $V\in T\pa X$, where $\nu$ is the outward pointing unit normal. (We thus allow for the possibility that parts of the boundary are totally geodesic.) Then the product manifold $M:=\R_t\times X$, $g=-dt^2+h$, is admissible, with $\pa M$ strictly null-convex if and only if $\pa X$ is strictly convex. See Figure~\ref{FigGExProd}; another example is (partially) shown in Figure~\ref{FigRSTrans}.

\begin{figure}[!ht]
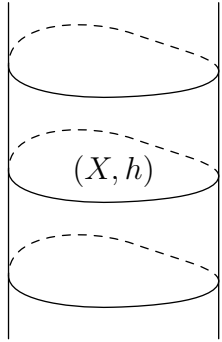

\centering
  \inclfig{GExProd}
  \caption{An admissible manifold, obtained as the product of the real line with a Riemannian manifold $X$ with convex boundary. We can allow parts of $\pa X$ to be flat.}
\label{FigGExProd}
\end{figure}

%%%%%%%%%%%%%%%%%%%%%%%%%%%%%%%%%%%%%%%%%%%%%%%%%%
\subsection{Broken null-geodesics}
\label{SubsecGeoG}

Throughout this section, $(M,g)$ will be a fixed admissible Lorentzian manifold. Motivated by the fact that singularities of solutions of wave equations on $(M,g)$ propagate along null-geodesics in $M^\circ$ and undergo reflection according to Snell's law at the boundary $\pa M$, see Taylor~\cite{TaylorReflection}, we rigorously define and study such \emph{broken null-geodesics} in this section. Define the open submanifold
\begin{equation}
\label{EqGeoGRedLM}
  \Lightb M := L M \setminus ( T^+_{\pa M}M \cup T\pa M )
\end{equation}
of $L M$, so $(p,V)\in\Lightb M$ if and only if $V\in L_p M$ and $V\in T^-_{\pa M}M$ when $p\in\pa M$. We then introduce:

\begin{definition}
\label{DefGeoG}
  Let $(p,V)\in\Lightb M$. We call a piecewise smooth curve $\gamma\colon I\to M$, with $0\in I\subset\R$ an open connected interval, $\gamma(0)=p$, $\gamma'(0)=V$, a \emph{broken null-geodesic}, if
  \begin{enumerate}
  \item for all open intervals $J\subset I$ with $\gamma(J)\cap\pa M=\emptyset$, $\gamma|_J$ is an affinely parameterized null-geodesic in $(M^\circ,g)$;
  \item if $s\in I$, $\gamma(s)\in\pa M$, then for small $\eps>0$, $\gamma|_{(s-\eps,s]}$ and $\gamma|_{[s,s+\eps)}$ are null-geodesics with $\gamma(s\pm(0,\eps))\subset M^\circ$, and $\gamma'(s+0)=\rho(\gamma'(s-0))$, where $\rho$ is the reflection map~\eqref{EqGeoARefl}.
  \end{enumerate}
\end{definition}

Thus, broken null-geodesics are null-geodesics which undergo reflection at the boundary $\pa M$ preserving their velocity tangent to $\pa M$; see Figure~\ref{FigGeoGBroken}.

\begin{figure}[!ht]
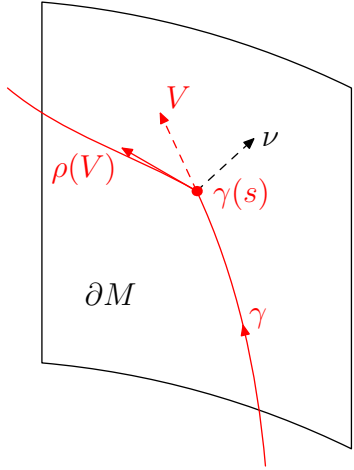

\centering
  \inclfig{GeoGBroken}
  \caption{A broken null-geodesic undergoing a reflection at $\gamma(s)\in\pa M$. Here, $V=\gamma'(s)$.}
\label{FigGeoGBroken}
\end{figure}

A broken null-geodesic with $(\gamma(0),\gamma'(0))=(p,V)$ as in this definition always exists on sufficiently small intervals $I=(-\eps,\eps)$, $\eps>0$: when $p\in M^\circ$, $\gamma|_I$ is an interior null-geodesic, while for $p\in\pa M$, one takes $\gamma(s)=\exp_p(s V)$ for $s\geq 0$ and $\gamma(s)=\exp_p(s\rho(V))$ for $s\leq 0$. Also note that if $\gamma_j\colon I_j\to M$, $j\in J$, $0\in I_j$, are broken null-geodesics which all have the same initial condition, then the prescription $\gamma|_{I_j}=\gamma_j$ defines a broken null-geodesic $\gamma\colon \bigcup_{j\in J}I_j\to M$. Thus, for $(p,V)$ as in the above definition, there always exists an \emph{inextendible broken null-geodesic} with initial position $p$ and speed $V$.

\begin{definition}
\label{DefGeoGExp}
  For $(p,V)\in\Lightb M$, let $\gamma\colon I\to M$ denote the unique inextendible broken null-geodesic with $(\gamma(0),\gamma'(0))=(p,V)$. Suppose $1\in I$. We then define the \emph{broken exponential map} by $\expb_p(V):=\gamma(1)$. Denote the domain of definition of $\expb$ by $\cD\subset\Lightb M$.
\end{definition}

Thus, for $(p,V)$ with $\gamma((0,1))\subset M^\circ$, we simply have $\expb_p(V)=\exp_p(V)$. We proceed to analyze the properties of inextendible broken null-geodesics. For convenience, we make our choice of the Riemannian metric $g^+$ on $\wt M$ more specific by demanding
\begin{equation}
\label{EqGeoGReflIso}
  g^+(\rho(V),\rho(V))=g^+(V,V),\ \ V\in T_p M,\ p\in\pa M.
\end{equation}
This is easily arranged by taking any Riemannian metric $g^+_0$ on $T_{\pa M}M$, then letting $g^+_1(V,W):=g^+_0(V,W)+g^+_0(\rho(V),\rho(W))$ for $V,W\in T_{\pa M}M$, and finally taking $g^+$ to be a Riemannian metric on $\wt M$ extending $g^+_1$ smoothly to the rest of $\wt M$. The consequence of \eqref{EqGeoGReflIso} is that the $g^+$-length of the tangent vector of a broken null-geodesic $\gamma$ is continuous when $\gamma$ hits the boundary.

\begin{prop}
\label{PropGeoGMax}
  Let $\gamma\colon I\to M$ be a broken null-geodesic with $\gamma'(0)\in L^+_p M$, and let $I_+:=\sup I\in\R\cup\{\infty\}$. Then $\gamma$ is future inextendible\footnote{By this we mean that the parameter interval for which the maximal broken null-geodesic with the same initial data as $\gamma$ is defined has supremum equal to $\sup I$.} if and only if one of the following happens:
  \begin{enumerate}
  \item\label{ItGeoGMaxT} $t(\gamma(s))\to\infty$ as $s\to I_+$.
  \item\label{ItGeoGMaxL} $I_+<\infty$, $t(\gamma(s))\to t_\infty<\infty$ as $s\to I_+$, and $I_+\in\ol{\gamma^{-1}(\pa M)}$.
  \end{enumerate}
  There exists an analogous characterization of past inextendibility.
\end{prop}

In other words, a broken null-geodesic is future-inextendible if and only if it leaves every region $\{t\leq t_0\}$, $t_0<\infty$ (this may happen even in the case $I_+<\infty$, e.g.\ for $M=\{(\tilde t,x)\colon|\tilde t|<\pi/2,|x|\leq 1\}$ with the metric $g=-dt^2+dx^2$, $t=\tan\tilde t$), or it undergoes infinitely many reflections as $s\to I_+<\infty$; similarly for past inextendibility. We remark that the latter scenario can indeed occur in certain cases when $\pa M$ is flat to infinite order; see \cite[\S6]{TaylorGrazing}.

\begin{proof}[Proof of Proposition~\ref{PropGeoGMax}]
  We note that $dt(\gamma'(s))>0$ for all $s\in I$ since $dt$ is past timelike and $\gamma'(s)$ is future causal; hence $t\circ\gamma$ is strictly increasing.

  If $I_+=\infty$, then $\gamma$ is clearly future inextendible. Suppose $I_+<\infty$ and $t(\gamma(s))\to\infty$ as $s\to I_+$. If there were an extension $\gamma_1\colon I_1\to M$, $I_+\in I_1$, of $\gamma$, then
  \[
    \infty>t(\gamma_1(I_+))=\lim_{\eps\to 0+} t(\gamma_1(I_+-\eps)) = \lim_{\eps\to 0+} t(\gamma(I_+-\eps)) = \infty,
  \]
  a contradiction. We next claim that
  \begin{equation}
  \label{EqGeoGMaxEscape}
    I_+=\infty\ \Lra\ t(\gamma(s))\to\infty,\ s\to\infty.
  \end{equation}
  Taking this for granted, the assumption $t(\gamma(s))\to t_\infty<\infty$ implies $I_+<\infty$; moreover, $\gamma(s)$ stays in a fixed compact set as $s\to I_+$ since $t$ is proper. If there exists a broken null-geodesic $\gamma_1\colon I_1\to M$ extending $\gamma$, $I_+\in I_1$, then since $\gamma_1^{-1}(\pa M)\subset I_1$ is discrete by definition, we infer that $I_+$ is not a limit point of $\gamma_1^{-1}(\pa M)$, hence not of $\gamma^{-1}(\pa M)$. Conversely, if $I_+\not\in\ol{\gamma^{-1}(\pa M)}$, then $s_0:=\max\{s\in I\colon \gamma(s)\in\pa M\}<I_+$. Let $(p_0,V_0):=(\gamma(s_0),\gamma'(s_0+0))$, then
  \[
    \gamma_1(s) := \wt\exp_{p_0}\bigl((s-s_0)V_0\bigr)
  \]
  satisfies $\gamma_1(s)=\gamma(s)\in M^\circ$ for $s\in(s_0,\infty)\cap I$. If $\gamma_1(I_+)\not\in\pa M$, then $\gamma_1|_{I\cup(s_0,I_++\eps)}$, defined for small $\eps>0$, is an extension of $\gamma$. Otherwise, $\gamma_1$ intersects $\pa M$ at $s=I_+$, and it necessarily does so \emph{transversally} according to Proposition~\ref{PropGeoANull}; hence we can continue $\gamma_1(s)$ past $s=s_1$ as a broken null-geodesic by defining
  \[
    \gamma_1(s_1+s'):=\wt\exp_{\gamma_1(s_1)}\bigl(s'\rho(\gamma_1'(s_1-0))\bigr),
  \]
  $s'>0$ small. This construction shows that $\gamma$ is future extendible past $I_+$.

  It remains to prove~\eqref{EqGeoGMaxEscape}. Assume to the contrary that
  \begin{equation}
  \label{EqGeoGTInfty}
    t(\gamma(s))\to t_\infty<\infty\ \ \text{as}\ \ s\to\infty.
  \end{equation}
  Since $K:=t^{-1}([t(\gamma(0)),t_\infty])$ is compact, there exists $c>0$ such that
  \begin{equation}
  \label{EqGeoGMaxBound}
    dt(V) \geq c|V|_{g^+}, \ \ V\in L^+_K M.
  \end{equation}
  Define $\ell(s):=|\gamma'(s)|_{g^+}$. Since the difference of connections $D=\nabla^{g^+}-\nabla^g$ induces a bilinear map $T_p M\otimes T_p M\to T_p M$, $D_p(X,Y)=\nabla^{g^+}_X Y-\nabla^g_X Y$, we can write for $s$ with $\gamma(s)\in M^\circ$:
  \[
    \frac{1}{2}\frac{d}{ds}\bigl(\ell(s)^2\bigr) = g^+|_{\gamma(s)}\bigl(\nabla^{g^+}_{\gamma'(s)}\gamma'(s),\gamma'(s)\bigr) = g^+|_{\gamma(s)}\bigl(D_{\gamma(s)}(\gamma'(s),\gamma'(s)),\gamma'(s)\bigr)
  \]
  On the other hand, if $\gamma(s)\in\pa M$, then $\ell(s+0)=\ell(s-0)$ in view of \eqref{EqGeoGReflIso}. Since $\gamma(s)\in K$ remains in a compact set, this implies
  \[
    \frac{1}{2}\frac{d}{ds}\bigl(\ell(s)^2\bigr) \geq -C\ell(s)^3
  \]
  where $C>0$ is a uniform constant only depending on $K$. Rewriting this differential inequality as $(1/\ell)'\leq C$, we obtain
  \begin{equation}
  \label{EqGeoGMaxLengthBound}
    \ell(s) \geq \frac{1}{\frac{1}{\ell(0)}+C s}.
  \end{equation}
  Therefore, the bound~\eqref{EqGeoGMaxBound} implies
  \[
    t(\gamma(s)) \geq t(\gamma(0)) + \int_0^s c\ell(u)\,du \geq t(\gamma(0)) + \frac{c}{C}\log\bigl(1+C s\ell(0)\bigr),
  \]
  which exceeds $t_\infty$ for sufficiently large $s$, contradicting \eqref{EqGeoGTInfty}. The proof is complete.
\end{proof}

We next study the regularity properties of the broken exponential map. For $(p,V)\in\cD$, the domain of definition of $\expb$, consider the maximal broken null-geodesic $\gamma(s)=\expb_p(s V)$, $s\in I$, let
\[
  \sfJ(p,V):=\#\{s>0\colon\gamma(s)\in\pa M\}\in\N_0\cup\{\infty\}
\]
denote the number of reflections at $\pa M$, and enumerate the affine parameters for which $\gamma$ intersects the boundary:
\[
  \gamma^{-1}(\pa M)\cap(0,\infty) =: \{\sfs_j(p,V) \colon j=1,\ldots,\sfJ(p,V)\},\ \ 0<\sfs_j(p,V)<\sfs_{j+1}(p,V).
\]
Further, let
\[
  \sfp_j(p,V):=\gamma(\sfs_j(p,V)),\ \ \sfV_j(p,V):=\gamma'(\sfs_j(p,V)+0) \in \Lightb_{\sfp_j(p,V)}M
\]
denote the position and the velocity of the broken null-geodesic leaving the boundary at a reflection point. For $k\in\N_0$, define
\[
  \cD_k^\circ := \bigl\{ (p,V)\in\cD \colon \#\{s\in(0,1)\colon \gamma(s)\in\pa M\}=k,\ \expb_p(V)\not\in\pa M\bigr\},
\]
i.e.\ $k$ is the number of reflections of the broken null-geodesic segment $\gamma((0,1))$. Let $\cD_k$ denote the closure of $\cD_k^\circ$ in $\cD$. See Figure~\ref{FigGeoGExpRefl}.

\begin{figure}[!ht]
\centering
  \inclfig{GeoGExpRefl}
  \caption{A broken null-geodesic with initial data $(p,V)\in\cD_2^\circ$. Here $\sfp_j=\sfp_j(p,V)=\expb_p(\sfs_j(p,V)V)$ and $\sfV_j=\sfV_j(p,V)$.}
\label{FigGeoGExpRefl}
\end{figure}

\begin{prop}
\label{PropGeoGExpCont}
  The broken null-geodesic flow on an admissible Lorentzian manifold $(M,g)$ has the following properties:
  \begin{enumerate}
  \item\label{ItGeoGExpContOpenk} For every $k\in\N_0$, the set $\cD_k^\circ\subset\Lightb M$ is open. The functions $\sfs_j(p,V)$ as well as the points $(\sfp_j(p,V),\sfV_j(p,V))\in\Lightb M$ depend smoothly on $(p,V)\in\cD_k^\circ$, and so does $\expb_p(V)\in M^\circ$.
  \item\label{ItGeoGExpCl} We have the decomposition $\pa\cD_k=B_{k,-}\sqcup B_{k,+}$ into a disjoint union of the closed sets
    \begin{align*}
      B_{k,-} &= \{ (p,V)\in\cD \colon \sfJ(p,V)\geq k,\ \sfs_k(p,V)=1 \}, \\
      B_{k,+} &= \{ (p,V)\in\cD \colon \sfJ(p,V)\geq k+1,\ \sfs_{k+1}(p,V)=1 \}.
    \end{align*}
    See Figure~\ref{FigGExpCl}. Furthermore, $\sfs_j$, $\sfp_j$, and $\sfV_j$ for $1\leq j\leq k$ as well as $\expb$ extend from $\cD_k^\circ$ to smooth functions on $\cD_k$.
  \item\label{ItGeoGExpContOpen} $\cD$ is open in $\Lightb M$; more precisely, $\cD_k^\circ\cup\bigcup_{j<k}\cD_j$ is open for all $k\in\N_0$. The map $\expb$ is continuous on $\cD$.
  \end{enumerate}
\end{prop}

\begin{figure}[!ht]
\centering
  \inclfig{GeoGExpCl}
  \caption{Illustration of part~\itref{ItGeoGExpCl} of Proposition~\ref{PropGeoGExpCont}. \textit{Left:} initial data $(p,V)\in B_{1,+}$. \textit{Right:} initial data $(p,V)\in B_{1,-}$, with $(\sfp_1,\sfV_1)$ defined by smooth extension from $\cD_1^\circ$.}
\label{FigGExpCl}
\end{figure}

\begin{proof}[Proof of Proposition~\ref{PropGeoGExpCont}]
  \ul{\itref{ItGeoGExpContOpenk}}: for $k=0$ and $(p,V)\in\cD_0^\circ$, we have $\expb_p(s V)=\wt\exp_p(s V)$ for $0\leq s\leq 1$, and $\wt\exp_p(V)\in M^\circ$; hence the smooth dependence of $\expb_p(V)$ on $(p,V)$ follows from that of the standard exponential map $\wt\exp_p(V)$.

  Consider now $k\geq 1$. Denote the map acting by dilation by $c\in\R$ in the fibers by $R_c\colon L M\to L M$, $(p,V)\mapsto (p, c V)$. Let $(p,V)\in\cD_k^\circ$, and let $\bar\sfs_j=\sfs_j(p,V)$, $\bar\sfp_j=\sfp_j(p,V)$, and $\bar\sfV_j=\sfV_j(p,V)$ for $j=1,\ldots,k$. We start by defining neighborhoods of $(\bar\sfp_j,\bar\sfV_j)$ of initial conditions of null-geodesics for which the next intersection with $\pa M$ is controlled. Thus, for $j=0$, let
  \[
    Z_0 := (\wt\exp\circ R_{\bar\sfs_1})^{-1}(\pa M).
  \]
  Note that $(p,V)\in Z_0$. Moreover, in a neighborhood of $(p,V)$, $Z_0$ is a smooth codimension 1 submanifold of $\Lightb M$ which is transversal to $\R_+ V\subset\Lightb_p M$; this follows from the implicit function theorem applied to the map $x\circ\wt\exp\circ R_{\bar s_1}$ where $x\in\CI(M)$ is a boundary defining function: this map has a non-zero differential at $(p,V)$ due to part~\itref{ItGeoANullOutward} of Proposition~\ref{PropGeoANull}. See Figure~\ref{FigGeoGExpContZ0}.

  \begin{figure}[!ht]
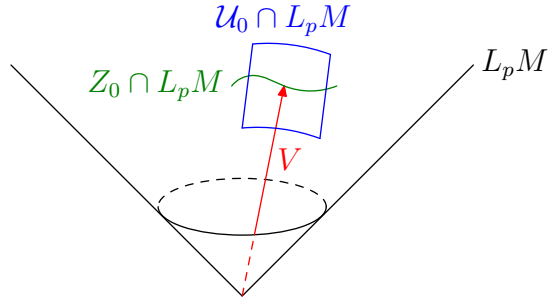

  \centering
    \inclfig{GeoGExpContZ0}
    \caption{Illustration of the proof of part~\itref{ItGeoGExpContOpenk} of Proposition~\ref{PropGeoGExpCont}. Shown are the preimage $Z_0$ of $\pa M$ under $\wt\exp\circ R_{\bar\sfs_1}$ within $L_p M$ (only the future half of which is drawn) near $V$, and the intersection of the neighborhood $\cU_0$ of $(p,V)$ with $L_p M$.}
  \label{FigGeoGExpContZ0}
  \end{figure}
  
  For a small neighborhood $\cU_0\subset\Lightb M$ of $(p,V)$ such that $\cU_0\subset\bigcup_{c\in(1-\eps,1+\eps)}R_c Z_0$, with $\eps>0$ small, and such that $Z_0\cap\ol{\cU_0}$ is a smooth connected submanifold transversal to all dilation (in the fiber) orbits intersecting $\cU_0$, define the function $d_0\in\CI(\cU_0)$ by
  \begin{equation}
  \label{EqGeoGExpContScaling}
    d_0(q,W) := \bar\sfs_1 c, \ \ \ R_c(q,W)\in Z_0,\ c\in(\tfrac{1}{1+\eps},\tfrac{1}{1-\eps}),
  \end{equation}
  so $d_0(p,V)=\bar\sfs_1$, and $\wt\exp_q(R_{d_0(q,W)}W)\in\pa M$ for $(q,W)\in\cU_0$. Similarly, but now working within $\pa M$, we define for $1\leq j\leq k-1$
  \[
    Z_j^\pa := \bigl(\wt\exp\circ R_{\bar\sfs_{j+1}-\bar\sfs_j}|_{\Lightb_{\pa M}M}\bigr)^{-1}(\pa M) \subset \Lightb_{\pa M}M,
  \]
  so $(\bar\sfp_j,\bar\sfV_j)\in Z_j^\pa$, and near this point, $Z_j^\pa$ is a smooth codimension 1 submanifold of $\Lightb_{\pa M}M$ transversal to $\R_+\bar\sfV_j\subset\Lightb_{\bar\sfp_j}M$. For a small neighborhood $\cU_j^\pa\subset\Lightb_{\pa M}M$ of $(\bar\sfp_j,\bar\sfV_j)$ such that $\cU_j^\pa\subset\bigcup_{c\in(1-\eps,1+\eps)}R_c Z_j^\pa$, with $Z_j^\pa\cap\ol{\cU_j^\pa}$ smooth and connected, define $d_j\in\CI(\cU_j^\pa)$ by
  \[
    d_j(q,W) = (\bar\sfs_{j+1}-\bar\sfs_j)c,\ \ \ R_c(q,W)\in Z_j^\pa,
  \]
  so $d_j(\bar\sfp_j,\bar\sfV_j)=\bar\sfs_{j+1}-\bar\sfs_j$, and $\wt\exp_q(R_{d_j(q,W)}W)\in\pa M$ for $(q,W)\in\cU_j^\pa$. Lastly, let $\cU_k^\pa\subset\Lightb_{\pa M}M$ denote a small neighborhood of $(\bar\sfp_k,\bar\sfV_k)$ such that $\cU_k^\pa\subset(\wt\exp\circ R_{1-\bar\sfs_k})^{-1}(M^\circ)$; in particular $(\bar\sfp_k,\bar\sfV_k)\in\cU_k^\pa$.

  We now construct a neighborhood of $(p,V)$ for which the $j$-th reflection point and velocity lie in $\cU_j^\pa$. Thus, encoding point and velocity of the extended manifold by the map
  \[
    \wt\exp'(q,W) := \bigl(\wt\exp_q(W), \tfrac{d}{ds}\wt\exp_q(s W)\bigr|_{s=1}\bigr)
  \]
  we inductively define $\cV_k^\pa:=\cU_k^\pa$ and, for $1\leq j\leq k-1$,
  \[
    V_j^\pa := \cU_j^\pa \cap \bigl(\rho\circ R_{d_j}^{-1}\circ\wt\exp'\circ R_{d_j}\bigr)^{-1}(V_{j+1}^\pa),
  \]
  where we used the reflection map $\rho$ defined in equation~\eqref{EqGeoARefl}, and finally
  \[
    \cV_0 := \cU_0 \cap \bigl(\rho\circ R_{d_0}^{-1}\circ\wt\exp'\circ R_{d_0}\bigr)^{-1}(V_1^\pa).
  \]
  Then $\cV_0$ is the desired neighborhood of $(p,V)$. Indeed, if $(q,W)\in\cV_0$, we inductively define $(q_0,W_0)=(q,W)$, and for $j=0,\ldots,k-1$
  \[
    (q_{j+1},W_{j+1}) := \rho\Bigl(R_{d_j(q_j,W_j)}^{-1}\bigl(\wt\exp'\bigl(R_{d_j(q_j,W_j)}(q_j,W_j)\bigr)\bigr)\Bigr) \in \cV_{j+1}^\pa.
  \]
  Then we have $\expb_q(\sfs_j(q,W) W)=q_j$ for $1\leq j\leq k$, where
  \[
    \sfs_j(q,W) = \sum_{i<j} d_i(q_i,W_i).
  \]
  Therefore $\expb_q(W)=\wt\exp_{q_k}\bigl((1-s_k(q,W))W_k\bigr)$, and by construction, we also have $(\sfp_j(q,W),\sfV_j(q,W))=(q_j,W_j)$, with smooth dependence on $(q,W)$.

  \ul{\itref{ItGeoGExpCl}}: Suppose $\cD\supset\pa\cD_k\ni(p,V)=\lim_{i\to\infty}(p_i,V_i)$ with $(p_i,V_i)\in\cD_k^\circ$, and denote $\bar\sfJ:=\sfJ(p,V)$, $\bar\sfs_j:=\sfs_j(p,V)$. The above arguments imply $\bar\sfJ\geq k-1$, and
  \[
    \bar\sfs_j = \lim_{i\to\infty} \sfs_j(p_i,V_i),\ \ j\leq\min(k,\bar\sfJ),
  \]
  likewise $(\bar\sfp_j,\bar\sfV_j):=(\sfp_j(p,V),\sfV_j(p,V))=\lim_{i\to\infty} (\sfp_j(p_i,V_i),\sfV_j(p_i,V_i))$ for these $j$. Let $\gamma(s)=\expb_p(s V)$ and $\gamma_i(s)=\expb_{p_i}(s V_i)$.
  
  Suppose first that $\bar\sfJ\geq k$, then $\bar\sfs_j\in(0,1]$ for $j\leq k$. If $\bar\sfs_k=1$, then $\gamma(s)$ undergoes $(k-1)$ reflections and ends at $\gamma(1)\in\pa M$, and $(p,V)\in B_{k,-}$. If $\sfs_k<1$, then both the case $\sfJ=k$ and the case $\sfJ\geq k+1$, $\sfs_{k+1}>1$, would imply $(p,V)\in\cD_k^\circ$. Hence, we must have $\sfJ\geq k+1$, and $\sfs_{k+1}\leq 1$. If $\sfs_{k+1}<1$, then the arguments for \itref{ItGeoGExpContOpenk} would imply that $\gamma_i((0,1))$ intersects $\pa M$ at least $k+1$ times for large $i$. Thus necessarily $\sfs_{k+1}=1$, and $(p,V)\in B_{k,+}$.
  
  In order to exclude the case that $\bar\sfJ=k-1$ ($k\geq 1$), note that $(p,V)\in\cD$ implies that we can define $\gamma(s)$ as a broken null-geodesic for $s\in[0,1+\eps]$ for some $\eps>0$. We claim that $\gamma$ necessarily has a $k$-th intersection point with $\pa M$, contradicting $\bar\sfJ<k$. If $k=1$, this is straightforward, as $\gamma|_{[0,1+\eps]}$ not intersecting $\pa M$ would imply (by continuity of $\wt\exp$) the same statement for $\gamma_i|_{[0,1+\eps]}$, contradicting $V_i\in\cD_1^\circ$. For $k\geq 2$, we note that
  \[
    \sfp_k(p_i,V_i) = \wt\exp_{\sfp_{k-1}(p_i,V_i)}\bigl((\sfs_k(p_i,V_i)-\sfs_{k-1}(p_i,V_i))\sfV_{k-1}(p_i,V_i)\bigr);
  \]
  passing to a subsequence, we may assume that $\sfs_k(p_i,V_i)\to\bar\sfs\in[\bar\sfs_{k-1},1]$. Since $\lim_{i\to\infty}\sfV_{k-1}(p_i,V_i)=\bar\sfV_{k-1}\in T^-_{\pa M}M$ is \emph{strictly} inward pointing (as follows from the definition of $\expb$ and $\cD$), there exists $c>0$ such that for all $i$, $\sfs_k(p_i,V_i)-\sfs_{k-1}(p_i,V_i)\geq c$; therefore $\bar\sfs>\bar\sfs_{k-1}$, and we obtain
  \[
    \wt\exp_{\bar\sfp_{k-1}(p,V)}\bigl((\sfs-\sfs_{k-1}(p,V))\sfV_{k-1}(p,V)\bigr) \in \pa M,
  \]
  so indeed $\bar\sfJ\geq k$ (and $\sfs_k(p,V)=\bar\sfs$).

  This proves the inclusion $\pa\cD_k\subset B_{k,-}\sqcup B_{k,+}$. (The disjointness of the two sets on the right is evident.) For the reverse inclusion, we note that $(p,V)\in B_{k,-}$ is the limit as $\eps\to 0$ of $(p,(1+\eps)V)\in\cD_k^\circ$ (this uses that $\sfs_j(p,c V)=c^{-1}\sfs_j(p,V)$ for $c>0$), while $(p,V)\in B_{k,+}$ is the limit of $(p,(1-\eps)V)\in\cD_k^\circ$. The smooth extendibility of $\sfs_j$ etc.\ follows easily from the construction used in the proof of part~\itref{ItGeoGExpContOpenk}.

  \ul{\itref{ItGeoGExpContOpen}}: Note that $B_{k,+}=B_{k+1,-}$, so the family $\expb|_{\cD_k}$, $k\in\N_0$, of smooth maps does glue to a continuous function on $\bigcup\cD_k$; furthermore, $\cD=\bigcup_{k\in\N_0}\cD_k$ by definition of broken null-geodesics. In view of \itref{ItGeoGExpContOpenk}, and noting that $B_{0,-}=\emptyset$, it remains to show that every $(p,V)\in B_{k-1,+}$, $k\in\N$, has an open neighborhood in $\Lightb M$ which is contained in $\cD_{k-1}\cup\cD_k^\circ$; but this follows again from the proof of part~\itref{ItGeoGExpContOpenk}.
\end{proof}

%%%%%%%%%%%%%%%%%%%%%%%%%%%%%%%%%%%%%%%%%%%%%%%%%%
\subsection{Tame broken null-geodesics}
\label{SubsecGeoTame}

We define the class of tame null-geodesics for which the possibility~\itref{ItGeoGMaxL} in Proposition~\ref{PropGeoGMax} does not occur for a given range of values of $t$:

\begin{definition}
\label{DefGeoTame}
  We call an inextendible broken null-geodesic $\gamma\colon I\to M$ \emph{tame for $-\infty\leq a<t<b\leq\infty$} if for all $a<a',b'<b$, we have $t(\gamma(I))\cap(a,a'),t(\gamma(I))\cap(b',b)\neq\emptyset$. If $\gamma$ is tame for $-\infty<t<\infty$, we simply say that $\gamma$ is \emph{tame}.
\end{definition}

By Proposition~\ref{PropGeoGMax} and its proof, this can be rephrased as follows: an inextendible broken null-geodesic $\gamma$ is tame for $a<t<b$ if and only if the only possible accumulation points of $\gamma^{-1}(\pa M)\cap(a,b)\subset\ol\R=\R\cup\{\pm\infty\}$ are $a$ and $b$; that is, $\gamma$ only undergoes a finite number of reflections whenever $t\circ\gamma$ stays in a fixed compact subset of $(a,b)$. Tame geodesics are precise those for which $t(\gamma(I))=\R$.

From the point of view of solving boundary value problems for wave equations, we have precise control over the singularities of geometric optics solutions along tame broken null-geodesics \cite{TaylorReflection}. There are much more general results about the propagation of singularities for boundary value problems, see for example \cite{TaylorGrazing,MelroseSjostrandSingBVPI,MelroseSjostrandSingBVPII,MelroseTaylorParametrix}, which would become relevant if one dropped the null-convexity assumption on $\pa M$. They in particular give rather precise information on the curves along which singularities intersecting the boundary tangentially propagate; for null-convex $\pa M$, these are null-geodesics within the boundary. Our reconstruction arguments on the other hand crucially rely on the \emph{spacelike} nature of the boundary light observation sets, which is guaranteed by the tameness assumption.

To illustrate Definition~\ref{DefGeoTame} and to provide a natural class of examples, we show:

\begin{prop}
\label{PropGeoTameStr}
  If $(M,g)$ is admissible with \emph{strictly} null-convex boundary, then all inextendible broken null-geodesics $\gamma\colon I\to M$ are tame.
\end{prop}

This is a generalization of \cite[Lemma~6.1]{TaylorGrazing}.

\begin{proof}[Proof of Proposition~\ref{PropGeoTameStr}]
  Assume the conclusion is false, then we must have $I_+:=\sup I<\infty$, and $\gamma^{-1}(\pa M)\cap[0,I_+)=\{b_j\colon j\in\N_0\}\subset I$, with
  \begin{equation}
  \label{EqGeoTameStrTimes}
    0\leq b_{j-1}<b_j\to I_+, \ \ t(\gamma(b_j))\to t_\infty<\infty.
  \end{equation}
  Denote $p_j:=\gamma(b_j)\in\pa M$. By the proof of Proposition~\ref{PropGeoGMax}, in particular the estimate~\eqref{EqGeoGMaxLengthBound}, there exists a constant $C_+>1$ such that $C_+^{-1}\leq |\gamma'(s)|_{g^+}\leq C_+$ for all $s\in[0,I_+)$; thus, $\gamma(s)$ is uniformly continuous, which implies that the limit $\lim_{s\to I_+}\gamma(s)=:p_\infty\in\pa M$ exists.
  
  Letting
  \[
    V_j:=\gamma'(b_j+0)\in\Lightbp_{p_j}M,
  \]
  we claim that $V_j$ converges to some $0\neq V_\infty\in L_{p_\infty}\pa M$, i.e.\ $V_\infty$ is tangent to the boundary; note that $C_+^{-1}\leq|V_j|_{g^+}\leq C_+$ for all $j$, proving that any subsequential limit of the $V_j$ must be a \emph{non-zero} element of $L\pa M$. To prove the convergence, denote by $\nu$ the outward unit normal to $\pa M$, and assume to the contrary that there is a subsequence $V_{j_k}$ such that $|g(V_{j_k},\nu)|\geq C_\pa>0$ for some fixed constant $C_\pa$. Using a finite number of local coordinate charts covering the compact set $K:=\pa M\cap t^{-1}([t(\gamma(0)),t_\infty])$, one easily sees that
  \begin{align*}
    C_R := \inf_{q\in K} \sup \{ t \colon \wt\exp_q((0,t)W) &\subset M^\circ\ \text{for all}\ W\in (T^-_{\pa M}M)_q, \\
      & C_+^{-1}\leq |W|_{g^+}\leq C_+,\ |g(W,\nu)|\geq C_\pa \}
  \end{align*}
  is \emph{positive}, as follows from the fact that in a local coordinate chart and for such $W$, we have $\wt\exp_q(s W)=q+s W+\cO(s^2)$, which does not return to $\pa M$ for a uniform amount of time (depending on $C_\pa,C_+,K$, and the $\cC^1(K)$ norm of the metric $g$). But then $b_{j_k+1}-b_{j_k}\geq C_R$, contradicting \eqref{EqGeoTameStrTimes}. A similar argument shows more generally that
  \begin{equation}
  \label{EqGeoTameStrGammaPrime}
    \lim_{s\to I_+}\gamma'(s)=V_\infty \in L_{p_\infty}\pa M.
  \end{equation}
  By affinely reparameterizing $\gamma$, we may assume $|V_\infty|_{g^+}=1$.

  Fix a boundary defining function $x$, and let
  \[
    f:= x\circ\gamma \geq 0,
  \]
  then $f$ is continuous on the closed interval $[0,I_+]$, with $f(b_j)=0$ for all $j$; therefore $\lim_{s\to I_+}f(s)=f(I_+)=0$. Let further
  \[
    \theta_j:=f'(b_j+0)=dx(V_j)>0,
  \]
  then $\lim_{j\to\infty}\theta_j=0$. We aim to prove estimates on the `chord lengths' $b_{j+1}-b_j$ and the `reflection angles' $\theta_j$ as $j\to\infty$ which will contradict the convergence \eqref{EqGeoTameStrTimes}; our arguments will slightly more generally prove that reflection points cannot accumulate near a strict null-convex boundary point.
  
  The strict null-convexity of $\pa M$ at $p_\infty$ implies, by continuity, that $(H x)(V,V)\leq-k<0$ for some constant $k>0$ whenever $V\in L_p\pa M$, $|V|_{g^+}=1$, $p$ near $p_\infty$. For large $j$ then, by~\eqref{EqGeoTameStrGammaPrime}, we have
  \[
    f''(s) = (H x)(\gamma'(s),\gamma'(s)) \leq -k/2,\ \ s\in(b_j,b_{j+1}),
  \]
  which gives an estimate for how close $\gamma$ stays to $\pa M$:
  \begin{equation}
  \label{EqGeoStrTameClose}
    f(s)\leq\theta_j(s-b_j)-k(s-b_j)^2/4\leq \theta_j^2/k.
  \end{equation}
  Furthermore, $f(b_{j+1})=0$ implies the estimate
  \begin{equation}
  \label{EqGeoStrTameClose2}
    b_{j+1}-b_j\leq 4\theta_j/k.
  \end{equation}
  Consider now a reflection point $p_j$, $j$ large, then $V_j$ is $\theta_j$-close to a null vector $V'_j\in L_{p_j}\pa M$. Let $k_j:=-(H x)(V'_j,V'_j)>0$, then the smoothness of $H x$ and the estimates~\eqref{EqGeoStrTameClose}--\eqref{EqGeoStrTameClose2} give
  \[
    f''(s) = -k_j + \cO(\theta_j + f(s) + |s-b_j|)=-k_j+\cO(\theta_j),\ \ s\in(b_{j-1},b_j).
  \]
  We also record that $k_j\geq k/2$ for large $j$. Therefore, for such $s$, we have
  \begin{align*}
    f'(s) &= \theta_j - (s-b_j)(k_j+\cO(\theta_j)), \\
    f(s) &= (s-b_j)\theta_j - (s-b_j)^2(k_j+\cO(\theta_j))/2.
  \end{align*}
  Hence, $f(b_{j+1})=0$ implies
  \begin{equation}
  \label{EqGeoTameStrBDiff}
    b_{j+1}-b_j = \frac{2\theta_j}{k_j+\cO(\theta_j)} \geq C_B\theta_j,
  \end{equation}
  and thus
  \[
    \theta_{j+1} = -f'(b_{j+1}) = -\theta_j + \frac{2\theta_j(k_j+\cO(\theta_j)}{k_j+\cO(\theta_j)} \geq \theta_j - C\theta_j^2,
  \]
  with $C$ and $C_B>0$ constants independent of $j$.

  Fix $j_0$ such that $\theta_j<\frac{1}{2 C}$ for $j\geq j_0$. Since $x\mapsto x-C x^2$ is increasing for $x<\frac{1}{2 C}$, we conclude that $\theta_j\geq a_j/C$, where $a_{j_0}=C \theta_{j_0}\in(0,1/2)$ and
  \[
    a_{j+1} = a_j - a_j^2.
  \]
  Since $a_j\geq\frac{C_A}{j-(j_0-1)}$ by Lemma~\ref{LemmaGeoTameSeq} below, the estimate~\eqref{EqGeoTameStrBDiff} implies that $b_j\geq C_0 + C_1\log j$ for some constants $C_0$ and $C_1>0$, contradicting \eqref{EqGeoTameStrTimes}. The proof is complete.
\end{proof}

\begin{lemma}
\label{LemmaGeoTameSeq}
  If $a_1\in(0,1/2)$ and $a_{j+1}=a_j-a_j^2$, then $a_j \geq C/j$ for some $C>0$.
\end{lemma}
\begin{proof}
  Clearly, $a_j>0$ for all $j$. Write $a_j = b_j / j$, then $b_j>0$, and
  \[
    b_{j+1} = b_j\left(1+\frac{1}{j}\Bigl(1-\frac{j+1}{j}b_j\Bigr)\right)
  \]
  If $b_j\leq\frac{j}{j+1}$ (this holds for  $j=1$), this gives $b_j\leq b_{j+1}\leq 1$. If on the other hand $\frac{j}{j+1}\leq b_j\leq 1$, then $\frac{j-1}{j}\leq b_{j+1}\leq 1$. Thus, $b_j\geq C$ for some $C>0$, as claimed.
\end{proof}

%%%%%%%%%%%%%%%%%%%%%%%%%%%%%%%%%%%%%%%%%%%%%%%%%%%%%%%%%%%%%%%%%%%%%%
\section{Reconstruction from boundary light observation sets}
\label{SecR}

In this section, we prove (a generalization of) Theorem~\ref{ThmIntro}, showing how one can reconstruct the topological, smooth, and conformal structure of suitable precompact subsets $S\Subset M$ from the observation of light cones on (subsets of) the null-convex boundary $\pa M$, following the arguments outlined in the introduction.

There are substantial differences compared to the arguments in \cite{KurylevLassasUhlmannSpacetime} due to the presence of a boundary which we will explain in more detail below: the boundary allows for the reconstruction of $S$ using (multiply) reflected rays; it necessitates certain restrictions on $S$ due to possible strong refocusing after reflection; and the codimension $1$ nature of $\pa M$ causes complications when there are null conjugate points on $\pa M$ --- we circumvent the latter by \emph{assuming} that there are no such conjugate points in the set $\cU\subset\pa M$ where we observe the future light cones from points in $S$.

Let $(M,g)$ denote an admissible manifold.

\begin{definition}
\label{DefRConj}
  Let $(q,V)\in\Lightb M$, see~\eqref{EqGeoGRedLM}, and suppose $p:=\expb_q(V)\in\pa M$. Then we say that $(q,V)$ and $p$ are \emph{not conjugate} if $\expb_q|_{L_q M}$ has injective differential at $V$, where we define the differential as the limit $D_{(1-\eps)V}\expb_q$ as $\eps\to 0+$.
\end{definition}

If $\expb_q(s V)\not\in\pa M$ for $0<s<1$, this can be phrased equivalently as the condition that the exponential map $\wt\exp_q|_{L_p M}$ has injective differential at $V$.

The existence of the limit follows from part~\itref{ItGeoGExpCl} of Proposition~\ref{PropGeoGExpCont}, since $(q,V)\in B_k^+$ for some $k\in\N_0$. Since broken null-geodesics are transversal to $\pa M$, we can rephrase the definition as follows: denote $Z=(\expb_q)^{-1}(\pa M)\subset\Lightb_q M$, which is a smooth codimension $1$ submanifold near $V$ (see the proof of Proposition~\ref{PropGeoGExpCont}). Then $(q,V)$ and $p$ are not conjugate if and only if the map $Z\ni W\mapsto\expb_q(W)\in\pa M$ has injective differential at $V$; that is, the boundary point near $p$ depends non-degenerately on the initial velocity. See also Figure~\ref{FigGeoGExpContZ0} for a closely related setting.

The implicit function theorem immediately gives:
\begin{lemma}
\label{LemmaRConj}
  If $(q,V)$ and $p$ are not conjugate, then, in the above notation, there exists a neighborhood $U\subset Z$ of $V$ such that $\expb_q|_U\colon U\to\expb_q(U)$ is a diffeomorphism onto its image, which is thus a 1-codimensional smooth submanifold of $\pa M$.
\end{lemma}

Recalling~\eqref{EqGeoGRedLM}, denote by
\[
  \Lightbp M := \Lightb M \cap L^+ M = L^+M \setminus (T^+_{\pa M}M\cup T\pa M)
\]
the set of future-directed light-like vectors which are inward pointing at the boundary. We then define by
\begin{equation}
\label{EqRFutLight}
  \cL_q^+ := \ol{\expb_q(\Lightbp_q M\cap\cD)}, \ \ q\in M,
\end{equation}
the future light cone from $q$. Thus, if $q\in M^\circ$, we simply have $\cL_q^+=\{q\}\cup\expb_q(L_q^+ M)$. (The set on the right hand side is already closed since $\expb_q$ is proper; this uses that there exists a global timelike function on $M$.)

\begin{thm}
\label{ThmR}
  Let $(M_j,g_j)$, $j=1,2$, be admissible Lorentzian manifolds, let $S_j\subset M_j^\circ$ be open with compact closure in $M_j$, and let $\cU_j'\Subset\cU_j\subset\pa M_j$ be open. Denote the collection of light observation sets by
  \begin{equation}
  \label{EqRObsSets}
    \cS_j := \{ \cL_q^+\cap\cU_j \colon q\in S_j \}.
  \end{equation}
  Assume:
  \begin{enumerate}
    \item \label{ItRDistinct} for any two points $q_1\neq q_2\in\ol{S_j}$, we have $\cL_{q_1}^+\cap\cU'_j\neq\cL_{q_2}^+\cap\cU'_j$.
    \item \label{ItRTame} all inextendible broken null-geodesics passing through a point in $\ol{S_j}$ are tame, see Definition~\ref{DefGeoTame};
    \item \label{ItRConj} for all $q\in S_j$ and $V\in\Lightbp_q M_j$ such that $p=\expb_q(V)\in\cU_j$, $(q,V)$ and $p$ are not conjugate.
  \end{enumerate}
  Suppose there exists a diffeomorphism $\Phi\colon\cU_1\xra{\cong}\cU_2$ such that
  \[
    \cS_2=\{\Phi(L)\colon L\in\cS_1\}.
  \]
  Then there exists a conformal diffeomorphism $\Psi\colon(S_1,g_1|_{S_1})\xra{\cong}(S_2,g_2|_{S_2})$.

  If in addition $\Phi$ is conformal, i.e.\ $\Phi^*(g_2|_{\cU_2})=f g_1|_{\cU_1}$ for some function $f\neq 0$, and preserves the time orientation, then $\Psi$ preserves the time orientation as well.
\end{thm}

In fact, we will show that the map $\Psi\colon S_1\to S_2$ given by the composition of $S_1\ni q\mapsto\cL_q^+\cap\cU_1\in\cS_1$, $\Phi$, and the inverse of $S_2\ni q\mapsto\cL_q^+\cap\cU_2\in\cS_2$ is a conformal (and time orientation preserving) diffeomorphism.

\begin{rmk}
\label{RmkRTame}
  For a general admissible manifold $(M,g)$, the constructions below allow for the reconstruction of $S$ from light observation sets if the closure $\bar S$ of the set $S$ of light sources as well as the subset of the boundary on which we observe are contained in a fixed slab $M':=t^{-1}((I_-,I_+))$, $-\infty\leq I_-<I_+\leq\infty$, with the property that all inextendible broken null-geodesics passing through a point in $\bar S$ are tame for $I_-<t<I_+$. One can then define a new time function $t'$, proper as a map $M'\to\R$, such that $t'\to\pm\infty$ as $t\to I_\pm$. Replacing $M$ by $M'$, condition~\itref{ItRTame} is satisfied.
\end{rmk}

\begin{rmk}
\label{RmkRUsefulRefl}
  Theorem~\ref{ThmR} allows for the reconstruction of subsets $S\subset M^\circ$ even in certain situations in which the first intersection point of future null-geodesics from sources in $S$ with $\pa M$ is not contained in $\cU$; that is, the theorem crucially uses possibly multiply reflected broken null-geodesics. As an example, in Figure~\ref{FigRSTrans}, one can take $S\subset M^\circ$ and $\cU\subset\pa M$ to be small neighborhoods of $q$ and $p$, respectively; if $\cU$ is sufficiently small, then the shown \emph{once broken} null-geodesics are the \emph{only} broken null-geodesics starting at $q$ and intersecting $\cU$.
\end{rmk}

Assumption~\itref{ItRDistinct} is very natural: we illustrate this with two examples.

\begin{example}
\label{ExRDistinct}
  Consider the cylinder $M_0=\{(t,r\omega)\colon r<1\}\subset\R\times\R^n$, $n\geq 1$, of radius $1$, see also equation~\eqref{EqGExM0}. Let $S_1=\{(t,r\omega)\colon |t|<1/2-r,\ r<1/2\}$ and $\cU=(0,2)\times\Sph^{n-1}$, $\cU'=[1/2,3/2]\times\Sph^{n-1}$. Then Theorem~\ref{ThmR} applies: the topological, differentiable, and conformal structure of $S_1$ can be recovered from the light observation sets $\cL_q^+\cap\cU$, $q\in S_1$.
  
  Denoting by $T_a\colon(t,x)\mapsto(t+a,x)$ the time translation operator, let $S_2:=S_1\cup T_{-2}(S_1)$. Using the notation \eqref{EqRObsSets}, we then have $\cS_1=\cS_2$, hence observers in $\cU$ cannot distinguish $S_1$ and $S_2$, even though the sets $S_1$ and $S_2$ are not homeomorphic ($S_1$ is connected, $S_2$ is not); assumption~\itref{ItRDistinct} is violated. See the left panel of Figure~\ref{FigRDistinct}.
\end{example}

\begin{figure}[!ht]
\centering
  \inclfig{RDistinct}
  \caption{Illustration of examples which violate assumption~\itref{ItRDistinct} of Theorem~\ref{ThmR}. \textit{Left:} example~\ref{ExRDistinct}. The sets $S_1$ and $S_2$ of sources have the same light observation sets in $\cU$, for instance $\cL_{q_1}^+\cap\cU=\cL_{q_2}^+\cap\cU$ for $q_1=(0,0)$ and $q_2=(-2,0)$. \textit{Right:} example~\ref{ExRDistinctCl}. All light observation sets from points in $S$ are distinct in $\cU$, but as $q\to(-2,0)$, the observation set $\cL_q^+\cap\cU$ converges to $\cL_{(0,0)}^+\cap\cU$. (The light cone based at $q$ does \emph{not} refocus near $(0,0)$ in three and more spacetime dimensions due to its distorted form, contrary to the appearance in this 2-dimensional picture.)}
\label{FigRDistinct}
\end{figure}

\begin{example}
\label{ExRDistinctCl}
  Consider $M_f\subset\R\times\R^n$, $n\geq 2$, defined in equation~\eqref{EqGExMf} with $R=1$, for the function $f(t,\omega)=\chi_0(t)\chi(\omega)$, where $\chi_0(t)\equiv 0$ for $t\geq 0$ and $\chi_0(t)=\delta e^{1/t}$ for $t\leq 0$, and where $\chi\in\CI(\Sph^{n-1})$ is identically $1$ in the neighborhood $|\omega-\omega_0|<1/2$ of some fixed $\omega_0\in\Sph^{n-1}\subset\R^n$, and $\chi(\omega)=0$ for $|\omega-\omega_0|>1$. See the right panel of Figure~\ref{FigRDistinct}. For $\delta>0$ sufficiently small, $M_f$ has a strictly null-convex boundary by Lemma~\ref{LemmaGExOp}. Let
  \begin{align*}
    S &= S' \cup S'', \\
      &\qquad S' = \{ (t,r\omega)\in\R^{1+n} \colon |t|<1/2-r,\ r<1/2 \}, \\
      &\qquad S''= \{ (t,r\omega)\in\R^{1+n} \colon |t+9/4|<1/4-r,\ r<1/4 \}.
  \end{align*}
  We use the observation set $\cU=(0,2)\times\Sph^{n-1}\subset\pa M_f$ and $\cU'=[1/2,3/2]\times\Sph^{n-1}$. Theorem~\ref{ThmR} applies to the set $S'$, and in fact yields a conformal diffeomorphism (which in this case is just the identity map on $\R^{1+n}$) between $(S',-dt^2+dx^2)$ and $(S_1,-dt^2+dx^2)$ from Example~\ref{ExRDistinct}. Theorem~\ref{ThmR} can also be shown (by a perturbative argument off the case $\delta=0$) to apply to $S''$ and $\cU$ for small $\delta>0$. If one attempts to recover $S$, all light observation sets $\cL_q^+\cap\cU$, $q\in S$, are distinct. However, we have
  \[
    \lim_{t\to -2} \cL_{(t,0)}\cap\cU \to \cL_{(0,0)}\cap\cU
  \]
  as smooth submanifolds of $\cU$. That is, separated points can have very similar light observation sets. This motivates the stronger hypothesis that light observation sets from points in the \emph{closure} $\bar S$ are distinct.
\end{example}

Fix an admissible Lorentzian manifold $(M,g)$, and sets $S\subset M$ and $\cU'\Subset\cU\subset\pa M$ satisfying the assumptions of Theorem~\ref{ThmR}, and denote $\bar\cS=\{\cL_q^+\cap\cU\colon q\in\bar S\}$. By assumption~\itref{ItRDistinct}, the map
\begin{equation}
\label{EqRLMap}
  L\colon \bar S\ni q\mapsto\cL_q^+\cap\cU\in\bar\cS
\end{equation}
is bijective, as is its restriction to $S$ as a map $S\to\cS$. There exists a unique topological, smooth, and conformal structure on $\cS$, defined by pushing these structures forward from $S\subset M$ to $\cS$ using $L$, which makes this map a conformal diffeomorphism. In order to prove Theorem~\ref{ThmR}, we need to show that we can uniquely recover these structures merely from the knowledge of the collection $\cS$ of subsets of the manifold $\cU$ and the conformal class of $g|_{\cU}$. \emph{From now on, we identify the set $S$ of sources and the set $\cS$ of light observation sets using the map~\eqref{EqRLMap}, and use the two interchangeably.}

The proof of Theorem~\ref{ThmR} will occupy the remainder of this section: in \S\ref{SubsecRT}, we show how to recover the topology of $\cS$, in \S\ref{SubsecRS} we recover the smooth structure, in \S\ref{SubsecRC} the conformal structure, and finally in \S\ref{SubsecRO} the time orientation of $S$.

%%%%%%%%%%%%%%%%%%%%%%%%%%%%%%%%%%%%%%%%%%%%%%%%%%
\subsection{Topology}
\label{SubsecRT}

We define a topology $\cT$ on $\cS$ by using the collection of sets of the form
\begin{align*}
  U_O &:= \{ L\in\cS \colon L\cap O\neq\emptyset \}, \ \ O\subset\cU\ \text{open},\\
  U^K &:= \{ L\in\cS \colon L\cap K=\emptyset \}, \ \ K\subset\cU\ \text{compact},
\end{align*}
as a subbasis. Note that the definition of $\cT$ only involves the \emph{set} $\cS$ and the a priori known topology of $\cU$.

\begin{prop}
\label{PropRT}
  The topology $\cT$ is equal to the subspace topology $\cT_M$ of $S\subset M$.
\end{prop}
\begin{proof}
  \ul{$\cT\subset\cT_M$}: We show that sets of the form $U_O$ and $U^K$ and open in $S\subset M$. If $O=\emptyset$, then $U_O=\emptyset$ is open. If on the other hand $O\neq\emptyset$, let $L=\cL_q^+\cap\cU\in\cS$, $q\in S$, and suppose $V\in\Lightbp_q M$ is such that $\expb_q(V)\in O$; in the notation of Proposition~\ref{PropGeoGExpCont}, we have $V\in B_{k,-}$ for some $k\in\N$, i.e.\ $p$ is the $k$-th intersection of the broken null-geodesic with initial data $(q,V)$ with the boundary $\pa M$, and $\sfp_k(V)=p$. By part~\itref{ItGeoGExpCl} of that proposition, $\sfp_k(W)$ depends continuously on $q'\in M$, $W\in\Lightb_{q'} M$, hence $\sfp_k(W)\in O$ when $(q',W)$ is close to $(q,V)$. This shows that $\cL_{q'}^+\cap O\neq\emptyset$, as desired.

  For $K\subset\cU$ compact, we claim that $S\setminus U^K$ is closed in the subspace topology of $M$: if $q_j\in S$, $\lim_{j\to\infty}q_j=:\bar q\in S$, and $V_j\in\Lightbp_{q_j}M$, $p_j:=\expb_{q_j}(V_j)\in K$, then, passing to a subsequence if necessary, we may assume that $p_j\to\bar p\in K$ as $j\to\infty$. Moreover, it follows from the proof of Proposition~\ref{PropGeoGMax}, see in particular the estimates~\eqref{EqGeoGMaxBound} and \eqref{EqGeoGMaxLengthBound}, that $|V_j|_{g^+}$ remains in a compact subset of $(0,\infty)$, hence we may assume that $(q_j,V_j)\to(\bar q,\bar V)\in\Lightbp M$. But by Proposition~\ref{PropGeoGExpCont}, we then have $\bar p=\expb_{\bar q}(\bar V)\in K$, hence $\bar q\not\in U^K$, as claimed.

  \ul{$\cT_M\subset\cT$}: we need to prove that for any $\cT_M$-open set $U\subset S$, every $q\in U$ has a $\cT$-open neighborhood which is contained in $U$. To see this, denote $L:=\cL_q^+\cap\ol{\cU'}$, and fix a compact set $K'$ with $\cU'\Subset K'\subset\cU$; for any $\eps>0$, let $K_\eps:=K'\setminus\bigl\{p\in\pa M\colon d_{g^+}(p,L)<\eps\bigr\}$, where $d_{g^+}(p,L)=\inf_{p'\in L}d_{g^+}(p,p')$ is defined using the Riemannian distance function of $g^+$. Further, pick a countable dense subset $\{p_i\colon i\in\N\}\subset\cL_q^+\cap\ol{\cU'}$, and let $O_{i,\eps}=\{p\in\cU\colon d_{g^+}(p,p_i)<\eps\}$. By the compactness of $L$, for each $\eps>0$, there exists a finite number $N(\eps)$ such that $L\subset\bigcup_{i=1}^{N(\eps)} O_{i,\eps}$. Consider now the nested sequence of $\cT$-open neighborhoods
  \[
    U_j := U^{K_{1/j}} \cap \bigcap_{i=1}^{N(1/j)} U_{O_{i,1/j}}
  \]
  of $\cL^+_q\cap\cU$.
  
  Suppose that $U_j\not\subset U$ for all $j$, then we can pick a sequence $q_j\in U_j\setminus U\subset S$, and we may assume without loss that $q_j\to\bar q\in\bar S$. It then follows that $\cL_{\bar q}^+\cap\ol{\cU'}$ is equal to the set of limit points of the sequence of sets $\cL_{q_j}^+\cap\ol{\cU'}$; for $\bar q\not\in\pa M$, this is a consequence of Proposition~\ref{PropGeoGExpCont}, while for $\bar q\in\pa M$, recalling the definition~\eqref{EqRFutLight}, this follows from a simple approximation argument. By definition of the sets $U^{K_{1/j}}$, we infer that $\cL_{\bar q}^+\cap\ol{\cU'}\subset L$. If this were a strict inclusion (of \emph{closed} sets), we could find $i_0\in\N$ with $p_{i_0}\in L\setminus\cL_{\bar q}^+$ and $j_0\in\N$ such that $O_{i_0,1/j}\subset L\setminus\cL_{\bar q}^+$ for all $j\geq j_0$. However, by definition of $U_{O_{i_0,1/j}}$, there exists, for all $j\geq j_0$, a point $x_j\in\cL_{q_j}^+\cap O_{i_0,1/j}$; hence $p_{i_0}=\lim_{j\to\infty}x_j$ is a limit point of the sets $\cL_{q_j}^+$, hence contained in $\cL^+_{\bar q}$, which is a contradiction. Therefore, $\cL_{\bar q}^+\cap\ol{\cU'}=L=\cL_q^+\cap\ol{\cU'}$. By assumption~\itref{ItRDistinct} of Theorem~\ref{ThmR}, this implies $\bar S\setminus U\ni\bar q=q\in U$. This contradiction shows that $U_j\subset U$ for sufficiently large $j$, and the proof is complete.
\end{proof}

\begin{example}
\label{ExREarliest}
  A key construction in \cite{KurylevLassasUhlmannSpacetime} is the earliest observation time along timelike curves in the observation region. We give an example to indicate why, without modifications as in \S\ref{SubsecRS} below, this is not as useful in the present setting. Consider the cylinder $M_0\subset\R^{1+n}$, $n\geq 1$, with radius $1$, see equation~\eqref{EqGExM0}, and consider the set
  \[
    S = \{ (t,r\omega)\in\R^{1+n} \colon |t+1|<1/2-r,\ r<1/2 \}.
  \]
  We observe in the set $\cU=(0,3)\times\Sph^{n-1}$. Thus,
  \[
    \cL^+_{(t,0)}\cap\cU
      =\begin{cases}
         \{t+3\}\times\Sph^{n-1}, & t\leq -1, \\
         \{t+1,t+3\}\times\Sph^{n-1}, & t>-1.
       \end{cases}
  \]
  Correspondingly, the earliest observation time of $\cL^+_{(t,0)}$ along the timelike curve $\gamma(s)=(s,\omega_0)$ (with $\omega_0\in\Sph^{n-1}$ fixed) within $\pa M_0$, defined by $s_\gamma(t):=\inf\{s\colon\gamma(s)\in\cL^+_{(t,0)}\}$, is discontinuous, namely $s_\gamma(t)=t+3$ for $t\leq -1$, and $s_\gamma(t)=t+1$ for $t>-1$.
\end{example}

%%%%%%%%%%%%%%%%%%%%%%%%%%%%%%%%%%%%%%%%%%%%%%%%%%
\subsection{Smooth structure}
\label{SubsecRS}

With the topology of $\cS$ at our disposal, the space of continuous maps from $\cS$ into any topological space is well-defined. In order to recover the structure of $\cS$ as an (open) smooth manifold, we will, in a neighborhood of any point $q\in S$, define a coordinate system by using `earliest observation times' along suitable curves passing through points where $\cL^+_q\cap\cU$ is a smooth submanifold.

\begin{lemma}
\label{LemmaRSFinite}
  For $q\in M$ and $p\in\pa M$, the number of different vectors $V\in\Lightb_q M$ for which $\expb_q(V)=p$ is finite.
\end{lemma}
\begin{proof}
  Note that all such $V$ have to be non-zero. If $V$ is such a vector, then $\expb_q(s V)=p$ only holds for $s=1$. Thus, it suffices to prove that there are only finitely many rays in $\Lightb_q M\setminus\{0\}$ whose image under $\expb_q$ passes through $p$. Since $(\Lightb_q M\setminus\{0\})/\R_+$ is a compact space, it suffices to prove that every ray $\R_+ V$ with $\expb_q(V)=p$ has a punctured neighborhood consisting of rays whose image under $\expb_q$ does not contain $p$. But this follows from Lemma~\ref{LemmaRConj} (using assumption~\itref{ItRConj} of Theorem~\ref{ThmR}).
\end{proof}

\begin{lemma}
\label{LemmaRSTransversal}
  Suppose that $q\in S$ and $p\in\pa M$ are such that for all $V\in\Lightbp_q M$ satisfying $\expb_q(V)=p$, $(q,V)$ and $p$ are not conjugate. Then there exist a neighborhood $O\ni p$, an integer $N<\infty$, and $N$ pairwise transversal smooth codimension $1$ submanifolds $L_j$ ($j=1,\ldots,N$) of $O$ such that $\cL_q^+\cap O=\bigcup_{j=1}^N L_j$.
\end{lemma}

See Figure~\ref{FigRSTrans}.

\begin{figure}[!ht]
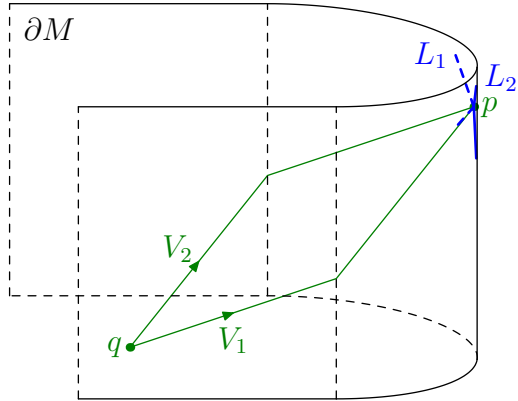

\centering
  \inclfig{RSTrans}
  \caption{Two different future-directed light rays from $q$ intersecting the boundary at the same point $p$. Under the assumption that $(q,V_j)$ and $p$ are not conjugate for $j=1,2$, the light observation set $\cL_q^+\cap\pa M$ is, near $p$, the union of two transversally intersecting codimension $1$ submanifolds $L_1,L_2\subset\pa M$.}
\label{FigRSTrans}
\end{figure}

\begin{proof}[Proof of Lemma~\ref{LemmaRSTransversal}]
  Let $\{V_1,\ldots,V_N\}:=(\expb_q)^{-1}(p)\cap\Lightbp_q M$. As in Lemma~\ref{LemmaRConj}, there exists a smooth codimension $1$ submanifold $Z_j\subset\Lightb_q M$ containing $V_j$ such that $L_j:=\expb_q(Z_j)$ is a smooth codimension $1$ submanifold of $\pa M$, which moreover is spacelike by Lemma~\ref{LemmaGeoANull}. Furthermore, by construction, $\bigcup_{j=1}^N L_j=\cL_q^+\cap O'$ for a sufficiently small neighborhood $O'$ of $p$.
  
  If $N=1$, we can take $O=O'$, and the proof is complete. If $N\geq 2$, we first establish the transversality of $L_j$ and $L_k$, $1\leq j\neq k\leq N$, at $p$. Let $\gamma_i(s)=\expb_q(s V_i)$, $i=j,k$; we then observe that $T_p L_i$ uniquely determines an outward lightlike ray through $p$, which is necessarily equal to the ray $\R_+W_i$, where $W_i:=\gamma_i'(1-0)\in (T_{\pa M}^+ M)_p$ for $i=j,k$; this uses Lemma~\ref{LemmaGeoANull}. Thus, if $T_p L_j=T_p L_k$, then $c W_j=W_k$ for some $c>0$. But then
  \begin{equation}
  \label{EqRSTransversalEq}
    \expb_p(-c s W_j)=\expb_p(-s W_k)\in\gamma_j([0,1])\cap\gamma_k([0,1]),\ \ s\in[0,1].
  \end{equation}
  Now for $s=1$, we have $\expb_p(-W_k)=q$, but we also have $\expb_p(-W_j)=q$ by construction. Thus, $c=1$, and by differentiating the equality in~\eqref{EqRSTransversalEq} in $s$ at $s=1$, we find $-V_j=-V_k$, contradicting $j\neq k$. The conclusion of the Lemma follows if we take $O\subset O'$ to be a sufficiently small neighborhood of $p$.
\end{proof}

Thus, away from a finite union of smooth codimension $2$ submanifolds of $\pa M$, $\cL_q^+\cap O$ is a smooth codimension $1$ submanifold of $\pa M$. Define the smooth part
\begin{equation}
\label{EqRSLReg}
\begin{split}
  \cL_q^\reg := \big\{ p \in \cL_q^+\cap\cU &\colon \text{there exists}\ p\in O\subset\pa M\ \text{open such that}\ \cL_q^+\cap O\ \text{is a} \\
  &\qquad \text{smooth connected codimension 1 submanifold of}\ O \bigr\}.
\end{split}
\end{equation}
(In the notation of the Lemma~\ref{LemmaRSTransversal}, we have $\cL_q^\reg\cap U=(\cL_q^+\cap U)\setminus\bigcup_{j\neq k}(L_j\cap L_k)$.)

Fix now any $q\in S$, and denote by $\mu\colon[-1,1]\to\cU$ a smooth curve in $\pa M$ which is transversal to $\cL_q^\reg$, with $\mu'(s)\neq 0$ for $s\in[-1,1]$, and such that $\mu(0)\in\cL_q^\reg$ and $\mu(s)\not\in\cL^+_q$ for $s\neq 0$. Consider the set
\[
  R'(\mu):=\{q'\in S\colon \#(\cL^\reg_{q'}\cap\mu([-1,1]))=1,\ \text{with transversal intersection}\}.
\]
While $R'(\mu)$ is neither open nor closed in general, it does contain an open neighborhood of $q$. Therefore, the set
\[
  R(\mu) := \bigcup_{\genfrac{}{}{0pt}{}{R\subset R'(\mu)}{\text{open in}\ S}} R
\]
is a non-empty open neighborhood of $q$. By part~\itref{ItGeoGExpCl} of Proposition~\ref{PropGeoGExpCont}, the earliest observation time
\[
  x^\mu \colon R(\mu)\ni q \mapsto s \in [-1,1],\ \ \text{where}\ \mu(s)\in\cL^+_q,
\]
is a smooth function on $R(\mu)$, and $x^\mu(q)=0$. (We stress that $R(\mu)$ and $x^\mu$ are well-defined given the topology of $S$ and the smooth structure of $\cU$.) We aim to show that suitable families of such functions $x^\mu$ give local coordinates near $q$. The key step is to show that there is always a large supply of curves $\mu$ for which $x^\mu$ is non-degenerate at $q$; more precisely:

\begin{lemma}
\label{LemmaRSGrad}
  Fix $q\in S$, and let
  \begin{align*}
    \frakM:=\{\mu\colon[-1,1]\xra{\CI}\cU \colon \mu\ &\text{is transversal to}\ \cL_q^+,\ \mu'(s)\neq 0\ \text{for}\ s\in[-1,1], \\
      & \qquad\qquad \mu(0)\in\cL_q^\reg,\ \mu(s)\not\in\cL_q^+\ \text{for}\ s\in[-1,1]\}.
  \end{align*}
  Then
  \[
    \bigcap_{\mu\in\frakM} \ker(dx^\mu|_q) = \{0\} \subset T_q M.
  \]
\end{lemma}

We give an analytic proof here, arguing by contradiction. The arguments in \S\ref{SubsecRO} below provide a different, more geometric proof.

\begin{proof}[Proof of Lemma~\ref{LemmaRSGrad}]
  Let $(-1,1)\ni r\mapsto q(r)\in S$ be a smooth path with $q(0)=q$ and $V:=q'(0)\neq 0\in T_q M$. Suppose that
  \begin{equation}
  \label{EqRSGrad0}
    dx^\mu(V)=0\ \text{for all}\ \mu\in\frakM;
  \end{equation}
  equivalently, for all $\mu\in\frakM$, the curve $r\mapsto\tilde\mu(r)$ defined by $\{\tilde\mu(r)\}=\cL_{q(r)}^+\cap\mu([-1,1])$ for small $r$, so $\tilde\mu(0)=\mu(0)$, satisfies $\tilde\mu'(0)=0$.
  
  Let now $O\subset\pa M$ be an open neighborhood of a point in $\cL^\reg_q$, as in \eqref{EqRSLReg}. Pick any non-empty $O'\Subset O$, and denote $L:=\cL_q^\reg\cap O'$. Since $\cL_q^\reg\cap O$ is smooth of codimension $1$, we can pick a smooth open map
  \[
    \ul\mu \colon L \times (-2,2) \to O
  \]
  such that $s\mapsto\ul\mu(p,s)$ is a curve with $\ul\mu(p,0)=p$, transversal to $\cL_q^\reg\cap O$, and so that $\ul\mu$ is a diffeomorphism onto its image $O''\subset\cU$.
  
  For small $r$, the preimage $\ul\mu^{-1}(\cL^+_{q(r)}\cap O'')$ is the graph of a smooth function $f(r,\cdot)\colon L\to(-2,2)$ and in fact $f\colon(-\eps,\eps)\times L\to(-2,2)$ (shrinking $\eps>0$ if necessary) is smooth. (These are consequences of Proposition~\ref{PropGeoGExpCont}.) Furthermore, $f(0,\cdot)\equiv 0$. Since we are assuming that \eqref{EqRSGrad0} holds, so $\pa_r f(0,p)\equiv 0$, the tangent space
  \[
    T(r,p) := T_{\ul\mu(p,f(r,p))}\cL^+_{q(r)}
  \]
  is $r^2$-close to $T(0,p)=T_p\cL^+_q$, uniformly for all $p\in O'$, hence the same is true for the unique future lightlike, outward pointing ray $\ell(r,p)\subset T(r,p)^\perp$, see Lemma~\ref{LemmaGeoANull}.
  
  Let now $V_1,V_2\in\Lightbp_q  M$ be two \emph{distinct} non-zero tangent vectors such that $p_j:=\expb_q(V_j)\in L$, $j=1,2$, then $\R_+\cdot\frac{d}{ds}\expb_q(s V_j)|_{s=1-0}=\ell(0,p_j)$. Denote by $W(r,p)$ a generator of $\ell(r,p)$ which depends smoothly on $(r,p)\in(-\eps,\eps)\times L$, and which is $r^2$-close to $W(0,p)$. Then the images of the two broken null-geodesics $s\mapsto\expb_{p_j}(-s W(0,p_j))$ for $j=1,2$ intersect \emph{cleanly} at $q$. But this implies that the point $q(r)$ is the \emph{unique} element near $q$ of the set of intersections of the broken null-geodesics $\expb_{\ul\mu(p_j,f(r,p_j))}(-s W(r,p_j))$, $j=1,2$, and moreover $q(r)$ depends smoothly on $f(r,p_j)$ and $W(r,p_j)$. The properties of $f$ and $W$ therefore imply that $q(r)$ is $r^2$-close to $q=q(0)$, contradicting the assumption $q'(0)\neq 0$ and completing the proof.
\end{proof}

In particular, for every $q\in S$, there exist $(n+1)$ curves $\mu_j\in\frakM$ such that the set $\{dx^{\mu_j}\colon j=0,\ldots,n\}$ is linearly independent at $q$, and therefore $(x^{\mu_j})_{j=0,\ldots,n}$ is a smooth local coordinate system near $q$. However, only knowing the collection $\cS$ of light observation sets, it is not immediately clear how to determine if a family $\mu_j$, $j=0,\ldots,n$, has this property. We thus argue indirectly: define a subalgebra
\[
  \cC\subset\cC^0(S)
\]
by declaring that $f\in\cC^0(S)$ is an element of $\cC$ if and only if for every $q\in S$, there exist an open neighborhood $U\ni q$ and curves $\mu_i\in\frakM$ (in the notation of Lemma~\ref{LemmaRSGrad}) for $0\leq i\leq n$ such that $U\subset\bigcap_{i=0}^n R(\mu_i)$, and a smooth function $F\colon\R^{n+1}\to\R$ so that
\begin{equation}
\label{EqRSFunc}
  f(q')=F(x^{\mu_0}(q'),\ldots,x^{\mu_n}(q')), \ \ q'\in U.
\end{equation}
By the arguments presented in this section, $\cC=\CI(S)$, and hence we have recovered the algebra of smooth functions on $S$ from the family of sets $\cS$.

Lastly then, a set of $n+1$ curves $\mu_i$, $0\leq i\leq n$, for which \emph{every} element of $\cC$ can be expressed in a neighborhood $U$ of $q$ in the form~\eqref{EqRSFunc} gives rise to a local coordinate system $(x^{\mu_j})_{j=0,\ldots,n}\colon U\to\R^{n+1}$. This completes the reconstruction of $S$ as a smooth manifold.

%%%%%%%%%%%%%%%%%%%%%%%%%%%%%%%%%%%%%%%%%%%%%%%%%%
\subsection{Conformal structure}
\label{SubsecRC}

The reconstruction of the conformal structure of $S$ is straightforward: if $q\in S$, let $V\in\Lightbp_q M$ be such that $p=\expb_q(V)\in\cL_q^\reg$, and put $L:=T_p(\cL_q^+\cap\cU)$. Consider the set
\[
  Q = \bigl\{ \mu\colon(-1,1) \to S \colon \mu\ \text{is smooth},\ \mu(0)=q,\ p\in\cL_{\mu(r)}^+\ \forall r,\ T_p(\cL_{\mu(r)}^+\cap\cU)=L \bigr\}
\]
of all paths which have the same outgoing future null ray at $p$, see Lemma~\ref{LemmaGeoANull}. Then $\{\mu'(0)\colon\mu\in Q\}=\R V\in T_q M$ recovers a 1-dimensional lightlike subspace of $T_q S$. Repeating this procedure for all points $p\in\cL_q^\reg$, and noting that $\cL_q^\reg\subset\cL_q\cap\cU$ is dense by Lemma~\ref{LemmaRSTransversal}, we can reconstruct an open subset of the light cone $L_q M\subset T_q M$. But $L_q M$ is a real-analytic submanifold of $T_q M$, hence this determines $L_q M$ uniquely. Since $q\in S$ was arbitrary, this proves that we can recover $L_S M$, hence the conformal structure of $S$. This finishes the proof of the first part of Theorem~\ref{ThmR}.

%%%%%%%%%%%%%%%%%%%%%%%%%%%%%%%%%%%%%%%%%%%%%%%%%%
\subsection{Time orientation}
\label{SubsecRO}

In order to recover the time orientation of $S$ when we are given the conformal structure of $\cU$ as well as its time orientation, we analyze the dependence of the boundary intersection point of a broken null-geodesic on its initial point:

\begin{lemma}
\label{LemmaRODep}
  Suppose $(-1,1)\ni r\mapsto (q(r),V(r))\in\Lightbp M$ is a smooth path such that $p(r):=\expb_{q(r)}V(r)\in\pa M$. Let $\gamma(r,s):=\expb_{q(r)}(s V(r))$ and $\gamma(s):=\gamma(0,s)$. Then
  \begin{equation}
  \label{EqRODep}
    g(q'(0),V(0)) = g(p'(0),\gamma'(1)).
  \end{equation}
\end{lemma}
\begin{proof}
  The values $0<\sfs_1(r)<\cdots<\sfs_k(r)=1$ of $s$ for which $\gamma(r,s)=0$ are smooth functions of $r$ for $r$ small, likewise the boundary intersection points $\sfp_j(r)=\gamma(r,\sfs_j(r))$; see also the discussion preceding Proposition~\ref{PropGeoGExpCont}. Define $\sfs_0(r):=0$ and $\sfp_0(r):=q(r)$. For $j=0,1,\ldots,k-1$, we then have $\int_{\sfs_j(r)}^{\sfs_{j+1}(r)} |\pa_s\gamma(r,s)|_g^2\,ds=0$ for all $r$, hence by differentiation in $r$, using that $\pa_s\gamma(r,s)$ is null for all $s$, and further using that $\gamma(r,s)$ is a null-geodesic for $s\in(\sfs_j(r),\sfs_{j+1}(r))$,
  \begin{align*}
    0 &= \int_{\sfs_j(r)}^{\sfs_{j+1}(r)} g(\pa_s\gamma(0,s),D_s\pa_r\gamma(r,s)|_{r=0})\,ds \\
      &= g\bigl(\gamma'(\sfs_{j+1}(r)-0), \sfp_{j+1}'(0)\bigr) - g\bigl(\gamma'(\sfs_j(r)+0), \sfp_j'(0)\bigr).
  \end{align*}
  Summing these identities and using that $\gamma'(\sfs_j(r)+0)-\gamma'(\sfs_j(r)-0)\perp T\pa M\ni\sfp_j'(0)$, all but the first and last terms cancel, and we obtain~\eqref{EqRODep}.
\end{proof}

Let now $(-1,1)\ni r\mapsto q(r)\in S$ be a timelike path; we show that one can determine whether $q$ is \emph{future} timelike:

\begin{prop}
\label{PropRO}
  Let $p(r)\in\cL_{q(r)}^\reg\cap\cU$ be a smooth path, and denote by $N\in T_{p(0)}\pa M$ the future-directed unit normal to the spacelike hypersurface $T_{p(0)}\cL_{q(0)}^+$. Then $q$ is future timelike if and only if $g(p'(0),N)<0$.
\end{prop}

We stress that this criterion only uses the conformal structure and time orientation of $(\cU,g|_{\cU})$.

\begin{proof}[Proof of Proposition~\ref{PropRO}]
  We claim that $p(r)=\expb_{q(r)}(V(r))$ with $V(r)\in\Lightbp_{q(r)}M$ smooth in $r$. By definition of $\cL_{q(0)}^\reg$, there exists a unique $V(0)\in\Lightbp_{q(0)}M$ such that $p(0)=\gamma(1)$ for $\gamma(s):=\expb_{q(0)}(s V(0))$. Similarly to the proof of Lemma~\ref{LemmaRSGrad}, let $W(r)$ denote a generator of the future-directed outward pointing null ray orthogonal to $\cL^+_{q(r)}\cap\cU$, see Lemma~\ref{LemmaGeoANull}, so that $W(0)=\gamma'(1)$. Since $q$ is timelike, the intersection of the broken null-geodesic $\mu_r(s):=\expb_{p(r)}(-s W(r))$ with $q$ is unique (if necessary shrinking the interval that $r$ takes values in) and clean; therefore we can choose a smooth function $s(r)$ such that $\mu_r(s(r))=q(r)$, with $s(0)=1$. But then $V(r)=-\mu_r'(s(r))/s(r)$ is smooth, as claimed.

  We can now apply Lemma~\ref{LemmaRODep} and use that the orthogonal projection of $\gamma'(1)=W(0)\in T_{p(0)}M$ to $T_{p(0)}\pa M$ is a positive multiple of $N$; since $V(0)$ is future-directed, we conclude that $q'(0)$ is future timelike iff $g(p'(0),N)<0$, proving the proposition.
\end{proof}

This finishes the proof of the second part of Theorem~\ref{ThmR}.

%%%%%%%%%%%%%%%%%%%%%%%%%%%%%%
\subsection*{Acknowledgments}

This research was partially conducted during the period P.H.\ served as a Clay Research Fellow. Moreover, P.H.\ gratefully acknowledges support from the Miller Institute at the University of California, Berkeley. G.U.\ is partially supported by the NSF grant DMS-1265958, a Si-Yuan Professorship at HKUST, a Simons Fellowship, and the Academy of Finland. The authors would like to thank an anonymous referee for carefully reading the manuscript and suggesting a number of improvements.

%%%%%%%%%%%%%%%%%%%%%%%%%%%%%%%%%%%%%%%%%%%%%%%%%%%%%%%%%%%%%%%%%%%%%%
\bibliographystyle{alpha}
%\bibliography{../bib/math,../bib/mathcheck,../bib/phys}

\end{document}